\documentclass{amsart}

\usepackage{amsmath}
\usepackage{amssymb}
\usepackage{amsthm}
\usepackage{graphicx}
\usepackage{hyperref}
\usepackage{nicefrac}
\usepackage{tikz-cd}
\usepackage[normalem]{ulem}

\usetikzlibrary{shapes}

\newcounter{statement}
\newtheorem{conjecture}[statement]{Conjecture}
\newtheorem{corollary}[statement]{Corollary}
\newtheorem{definition}[statement]{Definition}
\newtheorem{lemma}[statement]{Lemma}
\newtheorem{proposition}[statement]{Proposition}
\newtheorem{remark}[statement]{Remark}
\newtheorem{theorem}[statement]{Theorem}
\newtheorem*{theorem*}{Theorem}

\newcommand{\defeq}{\stackrel{\textnormal{def}}{=}}
\newcommand{\tn}{\textnormal}
\newcommand{\Aut}{\textnormal{Aut}}
\newcommand{\Gl}{\textnormal{GL}}
\newcommand{\im}{\operatorname{im}}
\newcommand{\Hom}{\textnormal{Hom}}
\newcommand{\Mcg}{\textnormal{Mcg}}
\newcommand{\Out}{\textnormal{Out}}
\newcommand{\onto}{\twoheadrightarrow}
\newcommand{\ad}{\textnormal{ad}}
\newcommand{\Stab}{\textnormal{Stab}}
\newcommand{\Torr}[1][1]{\mathcal{I}^{(#1)}}

\title[CSP for nilpotent gps and subsurfaces]{Congruence subgroup property for nilpotent groups and subsurface subgroups of Mapping Class Groups}
\author{Adam Klukowski}
\address{Mathematical Institute \\ University of Oxford \\ UK}
\email{klukowski@maths.ox.ac.uk}

\begin{document}

\begin{abstract}
We prove the Congruence Subgroup Property for two families of subgroups of Mapping Class Groups of finite-type surfaces. The first one is related to nilpotent quotients of the fundamental group and Johnson filtration, and along the way we give an elementary proof of a theorem of Ben-Ezra--Lubotzky. The second family consists of certain geometric subgroups obtained via subsurface inclusions.
\end{abstract}

\maketitle

\section{Introduction}

Many important groups arise from automorphisms of other groups. Examples include the families $\Gl_n\mathbb{Z}\curvearrowright\mathbb{Z}^n$, $\Aut(F_n)\curvearrowright F_n$, as well as the outer automorphism groups $\Out(F_n)$. Another family of examples are the Mapping Class Groups $\Mcg(\Sigma)$ for a surface $\Sigma$, which can be regarded as a subgroup of $\Out(\pi_1\Sigma)$ according to the theorem of Dehn-Nielsen-Baer \cite[Chapter 8]{primer_on_mapping_class_groups_FarbMargalit12}. In such situations, actions on finite characteristic quotients of the base group induce a system of finite-index subgroups, called \emph{congruence subgroups}, which induce the \emph{congruence topology}. A basic question towards understanding their algebraic structure is whether every finite-index subgroup arises this way, or equivalently, whether the profinite and pro-congruence topologies agree. This is referred to as having the \emph{Congruence Subgroup Property (CSP)}.

Groups possessing the Congruence Subgroup Property include $\tn{SL}_n\mathbb{Z}$ for $n\geqslant 3$, by the work of Mennicke \cite{finite_factor_gps_unimodular_gp_Mennicke65} and Bass--Lazard--Milnor--Serre \cite{sousgroupes_indice_fini_dans_slnZ_BassLazardSerre64}\cite{solution_csp_sln_sp2n_BassMilnorSerre67} (see also \cite{developments_csp_after_BMS_PrasadRapinchuk10}). Another two groups with CSP are $\Aut(F_2)$ and $\Out(F_2)$, by the work of Asada \cite{faithfulness_monodromy_reps_certain_alg_curves_Asada01} (see also \cite{csp_for_autf2_grouptheoretic_pf_of_Asadas_thm_BuxErshovRapinchuk11}\cite{csp_low_rk_free_free_metabelian_gps_BenEzraLubotzky18}). An example where CSP fails is $\textnormal{SL}_2\mathbb{Z}$ (this is one of the consequences of its virtual freeness). This is somewhat surprising, given that the natural map $\Out(F_2)\to\tn{GL}_2\mathbb{Z}$ is an isomorphism; the explanation is that the actions of this group on $F_2$ and $\mathbb{Z}^2$ induce different families of congruence subgroups.

The main focus of this article is the Congruence Subgroup Property of Mapping Class Groups of finite-type surfaces. CSP was proven in genus 0 by Diaz-Donagi-Harbater \cite{every_curve_Hurwitz_space_DiazDonagiHarbater89} (see also \cite{congruence_subgroup_problem_for_pure_braid_groups_McReynolds10}), in genus 1 with punctures by Asada \cite{faithfulness_monodromy_reps_certain_alg_curves_Asada01}\cite{csp_for_autf2_grouptheoretic_pf_of_Asadas_thm_BuxErshovRapinchuk11} (with no punctures, the Mapping Class Group is $\tn{SL}_2\mathbb{Z}$ and the fundamental group is $\mathbb{Z}^2$, and we mentioned that CSP fails in this case), and in genus 2 by Boggi \cite{csp_for_hyperelliptic_modular_gp_Boggi09}. The question of CSP in genus at least 3 is open, and mentioned in the Kirby List \cite[Problem 2.10]{problems_lowdim_top_Kirby97}.
\begin{conjecture}\label{conj:mcg_csp}
Let $\Sigma$ be a surface of genus at least 3. Then every finite-index subgroup of $\Mcg(\Sigma)$ is a congruence subgroup.
\end{conjecture}
One motivation for this question is its link to the study of profinite rigidity of 3-manifold groups. The most notorious open problem in this area was stated by Bridson--Reid \cite[Conjecture 2.1]{profinite_rigidity_free_gps_Bridson2023}. Below we present a certain weakening of their conjecture (compare \cite[Corollary B]{csp_for_mcgs_and_residual_finiteness_hyp_gps_Wilton24}).
\begin{conjecture}\label{conj:pi_comm}
Suppose $\mathcal{M},\mathcal{N}$ are hyperbolic 3-manifolds such that profinite completions of their fundamental groups are isomorphic. Then $\mathcal{M},\mathcal{N}$ are commensurable, meaning that they share a common finite-degree covering space.
\end{conjecture}
Positive answer to Conjecture \ref{conj:pi_comm} is known in a few special cases \cite{profinite_rigidity_hyp_4pctured_sphere_bundles_CheethamWest23}, which among other results rely on knowing CSP of low-genus Mapping Class Groups. In general, only almost profinite rigidity is known \cite{finite_vol_hyp_3mflds_almost_determined_finite_quots_Liu03}, meaning that hyperbolic 3-manifold groups with a prescribed profinite completion fall into finitely many commensurability classes. Apart from profinite rigidity, understanding finite-index subgroups would also be beneficial for a number of other problems concerning MCGs, such as the Ivanov conjecture \cite[Problem 2.11]{problems_lowdim_top_Kirby97} on virtual fibering of the moduli space, or property (T).

Usually, CSP is thought of as a property of the automorphism group. However, our approach is to look at analogous statements about simpler subgroups, similar to the philosophy of \cite{congruence_subgroup_problem_for_pure_braid_groups_McReynolds10} dealing with point-pushing subgroups and \cite{csp_symmetric_mcgs_Boggi24} investigating centralisers of finite subgroups. For these purposes it will be useful to have a more flexible definition.
\begin{definition}\label{def:csp}
Let $\lambda\colon\Lambda\to\Out(G)$ be a group homomorphism. A \emph{congruence subgroup} is any subgroup of $\Lambda$ of the form $(q_\star\circ\lambda)^{-1}(H)$, where $q\colon G\onto Q$ is a finite characteristic quotient of $G$, $q_\star\colon\Out(G)\to\Out(Q)$ is the induced map on outer automorphism groups, and $H\leq\Out(Q)$ is any subgroup. If $H$ is trivial, then its preimage is called \emph{principal}.

We say that $\lambda$ has the \emph{congruence subgroup property (CSP)} if every finite-index subgroup of $\Lambda$ containing $\ker\lambda$ is a congruence subgroup.
\end{definition}
Note that Definition \ref{def:csp} recovers the usual definition when $\lambda$ is the identity homomorphism; also, Conjecture \ref{conj:mcg_csp} can be phrased as the natural map $\Mcg(\Sigma)\to\Out(\pi_1)$ having CSP. It is not hard to see that every subgroup containing a principal congruence subgroup is itself congruence.

This article contains two main theorems. The first one, Theorem \ref{thm:csp_nilpotent}, holds for all nilpotent groups, and says that CSP can be determined by the action on abelianisation. This result is originally due to Ben-Ezra--Lubotzky \cite{csp_fg_nilp_gps_BenEzraLubotzky22}; we present another proof that is elementary and requires no background theory. We use Theorem \ref{thm:csp_nilpotent} to show that finite-index subgroups of Mapping Class Groups containing some term of the Johnson filtration are congruence subgroups. The second result, Theorem \ref{thm:csp_subsurface}, asserts that CSP of Mapping Class Groups behaves well under gluing smaller surfaces along their boundaries. As corollaries we show that many natural geometric subgroups of Mapping Class Groups have the Congruence Subgroup Property.

\subsection{Outer automorphisms of nilpotent groups}

The first theorem of this article is that outer automorphism groups of nilpotent groups possess the congruence subgroup property, as long as their abelianisations do.
\begin{theorem}\label{thm:csp_nilpotent}
Suppose $G$ is a finitely generated nilpotent group, and $\Lambda\leq\Out(G)$. If the linear action on abelianisation $\Lambda\to\Out(G_\textnormal{ab})$ has CSP, then $\Lambda\to\Out(G)$ has CSP.
\end{theorem}
This result (up to torsion) was first obtained by Ben-Ezra and Lubotzky \cite[Theorem 1.1]{csp_fg_nilp_gps_BenEzraLubotzky22}, basing on the theory of Malcev completion and unipotent groups and knowledge of CSP for solvable arithmetic groups. In Section \ref{ssec:pf_csp_nilpotent} we give a self-contained proof of Theorem \ref{thm:csp_nilpotent} involving only induction and elementary algebra.

\begin{remark}
Metabelian groups behave rather differently in terms of CSP \cite{csp_low_rk_free_free_metabelian_gps_BenEzraLubotzky18}.
\end{remark}

As an application of Theorem \ref{thm:csp_nilpotent}, we relate CSP to the Johnson filtration $\Torr[k]$ (see Definition \ref{def:AJ_filtration}).
\begin{corollary}\label{cor:congruence_Johnson}
If $\Gamma\leq\Mcg(\Sigma)$ is a finite-index subgroup containing the Johnson subgroup $\Torr[k]$ for some $k$, then $\Gamma$ is a congruence subgroup.
\end{corollary}
It was proved by Leiniger and McReynolds that solvable subgroups of Mapping Class Groups are separable \cite[Theorem 1.1]{separable_subgps_mcgs_LeinigerMcReynolds07}. More careful examination of their proof reveals that the quotients they produce actually annihilate some term of the Johnson filtration. In this situation Corollary \ref{cor:congruence_Johnson} applies, and upgrades separability to congruence separability.
\begin{corollary}
Solvable subgroups of $\Mcg(\Sigma)$ are congruence separable. Explicitly, for any solvable $\Gamma\leq\Mcg(\Sigma)$ and a mapping class $\varphi\notin\Gamma$ there exists a congruence subgroup $\Delta\leq\Mcg(\Sigma)$ such that $\Gamma\subseteq\Delta$ and $\varphi\notin\Delta$.
\end{corollary}

\subsection{Subsurface subgroups of MCGs}

The second result of this article asserts that CSP behaves well under inclusion of subsurfaces.
\begin{theorem}\label{thm:csp_subsurface}
Let $\Sigma$ be a possibly punctured surface with no boundary, and let $\Sigma=\bigcup_i\Sigma_i$ be its decomposition into disjoint subsurfaces glued along boundaries. Suppose that $\Lambda_i\leq\Mcg(\Sigma_i)$ are subgroups supported on the subsurfaces, and let $\Lambda=\prod_i\Lambda_i$ be the subgroup of $\Mcg(\Sigma)$ obtained from them by inclusion. If all $\Lambda_i\to\Out(\pi_1\Sigma_i)$ have the CSP, then $\Lambda\to\Out(\pi_1\Sigma)$ has the CSP.
\end{theorem}
Theorem \ref{thm:csp_subsurface} is proven in Section \ref{ssec:pf_csp_subsurface}.

As a first application of Theorem \ref{thm:csp_subsurface}, we extend the main result of the excellent survey of McReynolds \cite{congruence_subgroup_problem_for_pure_braid_groups_McReynolds10} from point-pushing to handle-pushing subgroups (see Section \ref{sec:handlepushing} for definition).
\begin{corollary}\label{corr:csp_handlepush}
Handle-pushing subgroups of Mapping Class Groups have the congruence subgroup property.
\end{corollary}

As a second application, using the work of Boggi \cite{csp_for_hyperelliptic_modular_gp_Boggi09}, we deduce the CSP of stabilisers of large simplices in the curve complex.
\begin{corollary}
Let $\mathcal{B}$ be a tuple of disjoint essential curves, such that every component of $\Sigma\setminus\bigcup \mathcal{B}$ has genus at most 2. Then $\Stab_{\Mcg(\Sigma)}(\mathcal{B})\to\Out(\pi_1\Sigma)$ has CSP.
\end{corollary}
Recall that there are finitely many orbits of simplices in the curve complex under the action of the Mapping Class Group (any collections of curves, whose complement has the same topological type, are related by a mapping class -- see \cite[Section 1.3]{primer_on_mapping_class_groups_FarbMargalit12}). Algebraically this means that there are only finitely many conjugacy classes of handle-pushing subgroups in the Mapping Class Group. This observation allows us to deduce the following:
\begin{corollary}
Let $\Gamma\leq\Mcg(\Sigma)$ be a finite-index subgroup. There exists a finite characteristic quotient $q\colon\pi_1\Sigma\to Q$ with the following property. For every collection $\mathcal{B}$ of simple closed curves, such that all components of the complement of $\mathcal{B}$ have genus at most 2, we have $\left(\ker q_\star\right)\cap\tn{Stab}_{\Mcg(\Sigma)}(\mathcal{B})\subseteq\Gamma$.
\end{corollary}

As a third application, we conditionally prove that curve stabilisers have CSP.
\begin{corollary}
Let $\beta$ be a non-separating simple closed curve on $\Sigma_g$. If $\Mcg(\Sigma_{g-1})$ has the CSP, then $\Stab(\beta)\leq\Mcg(\Sigma_g)$ has the CSP.
\end{corollary}

Using Theorem \ref{thm:csp_subsurface}, a proof of CSP for MCGs would follow if one could prove the following:
\begin{conjecture}\label{conj:curve_orbit_separation}
Let $\Gamma$ be a finite-index subgroup of $\Mcg(\Sigma)$, and $\alpha$ a simple closed curve on $\Sigma$. Then there exists a congruence subgroup $\Delta\leq\Mcg(\Sigma)$ such that the containment of orbits $\Delta.\alpha\subseteq\Gamma.\alpha$ holds.
\end{conjecture}
Notably, an analogue of Conjecture \ref{conj:curve_orbit_separation} appears in the recent independent work of Wilton and Sisto \cite{csp_for_mcgs_and_residual_finiteness_hyp_gps_Wilton24}. They prove a similar statement for a filling $\alpha$, conditional on every hyperbolic group being residually finite.

\subsection{Outline}

In Section \ref{sec:csp_ses} we describe an elementary lemma for proving CSP from simpler assumptions, which underlies proofs of both Theorem \ref{thm:csp_subsurface} and \ref{thm:csp_nilpotent}. It is adapted from \cite{congruence_subgroup_problem_for_pure_braid_groups_McReynolds10}.

In Section \ref{sec:Johnson_homs} we review the basic theory of filtrations on automorphism groups of nilpotent groups. We construct the key tool of Johnson homomorphisms.

Section \ref{sec:csp_nilpotent} is aimed at proving Theorem \ref{thm:csp_nilpotent}. We start in Subsection \ref{ssec:auxiliary} by proving two lemmata which supply quotients of nilpotent groups. Proof of \ref{thm:csp_nilpotent} is completed at the end of Subsection \ref{ssec:pf_csp_nilpotent}, after the main technical result of Proposition \ref{prop:csp_lastterm}.

Section \ref{sec:pf_csp_subsurface} is devoted to the proof of Theorem \ref{thm:csp_subsurface}. We prove some preliminary results in Subsections \ref{ssec:bdry_twists} and \ref{ssec:virt_retraction}, and complete the proof of Theorem \ref{thm:csp_subsurface} at the end of Subsection \ref{ssec:pf_csp_subsurface}, after going through the main technical part in Proposition \ref{prop:subsurface_containment}.

Section \ref{sec:handlepushing} is intended as an accessible illustration of Theorem \ref{thm:csp_subsurface}. There we give an informal introduction to point-pushing and handle-pushing subgroups, and explain that Theorem \ref{thm:csp_subsurface} applies to them.

\section*{Acknowledgements} The author is grateful to Dawid Kielak for suggesting this project and for his support throughout. He would also like to thank David El-Chai Ben-Ezra and Alex Lubotzky for valuable comments on an earlier version of this manuscript. This work was supported by the \textsl{Simons Investigator Award} 409745 from the Simons Foundation.

\section{A composition lemma for CSP}\label{sec:csp_ses}

In this section we present the key tool to prove the CSP with a divide-and-conquer approach. This lemma appears in \cite[Proposition 3.2]{congruence_subgroup_problem_for_pure_braid_groups_McReynolds10}, in a more restricted formulation. Our proof uses the same strategy, but we write it in detail for the sake of completeness.

\begin{lemma}\label{lmm:csp_ses}
Let $\Lambda\to\Out(G)$ be a group homomorphism, with $\Lambda$ finitely generated. Suppose there is a subgroup $\Delta\leq\Lambda$ such that
\begin{enumerate}
\item\label{slmm:ker} the restricted homomorphism $\Delta\to\Out(G)$ has CSP
\item\label{slmm:im} every finite-index subgroup of $\Lambda$ containing $\Delta$ is a congruence subgroup.
\end{enumerate}
Then $\Lambda\to\Out(G)$ has CSP.
\end{lemma}
\begin{proof}[Sketch of proof of Lemma \ref{lmm:csp_ses}]
Consider any finite-index $\Gamma\leq\Lambda$ containing $\ker(\Lambda\to\Out(G))$. Without loss of generality we can assume that $\Gamma\triangleleft\Lambda$. Then the intersection $\Gamma\cap\Delta$ is a finite-index subgroup of $\Delta$. Assumption (\ref*{slmm:ker}) tells us that there exists a finite characteristic quotient $q\colon G\onto Q$ satisfying
\begin{equation}\label{eq:ker_containment}
\Delta\cap\ker q_\star\subseteq\Gamma.
\end{equation}

Now consider $\Delta\cdot(\Gamma\cap\ker q_\star)$. This is a finite-index subgroup of $\Lambda$ containing $\Delta$. By assumption (\ref*{slmm:im}), there exists a finite characteristic quotient $r\colon G\onto R$ such that
\begin{equation}\label{eq:im_containment}
\ker r_\star\cap\Lambda\subseteq \Delta\cdot(\Gamma\cap\ker q_\star).
\end{equation}

Let $s\colon G\to S$ be the pullback (or fibre product) of $q$ and $r$. Explicitly, we consider the map $q\times r\colon G\to Q\times R$, and define $S=\im(q\times r)$ to be the image of $G$ inside $Q\times R$, and $s\colon G\onto S$ as the corresponding factorisation. We claim that the principal congruence subgroup induced by $(s,S)$ is contained in $\Gamma$.

To prove the claim consider any $\varphi\in\ker s$. The quotient $r\colon G\to R$ factors through $s\colon G\to S$, so $\varphi$ acts on $R$ by an inner automorphism. By the defining property (\ref{eq:im_containment}) of $(r,R)$, this means that $\varphi\in \Delta\cdot(\Gamma\cap\ker q_\star)$. Write $\varphi=\delta\gamma$ for some $\delta\in\Delta$, $\gamma\in\Gamma\cap\ker q_\star$. Since both $\varphi,\gamma$ belong to $\Lambda\cap\ker q_\star$, we must also have $\delta\in\Lambda\cap\ker q_\star$. But from the defining property (\ref{eq:ker_containment}) of $(q,Q)$ it follows that $\delta\in\Gamma$. Since both $\delta,\gamma$ are elements of $\Gamma$, then so is their product $\varphi$. This proves the claim.
\end{proof}

\section{Background on automorphisms of nilpotent groups}\label{sec:Johnson_homs}

Here we review the basic facts about the filtration on automorphism groups of nilpotent groups. This theory was used by Andreadakis to study the automorphisms of free groups \cite{automorphisms_free_grps_and_free_nil_gps_Andreadakis65}, and by Johnson for studying the Torelli subgroup of mapping class groups \cite{abelian_quot_mcg_Johnson80} (free and surface groups are only residually nilpotent, but that is good enough).

The \emph{lower central series} $\gamma_iG$ of a group $G$ is defined inductively via $\gamma_1G=G$ and $\gamma_{i+1}G=[G,\gamma_iG]$. We say that $G$ is \emph{nilpotent of class $K$} when $\gamma_KG\neq\gamma_{K+1}G=\{1\}$. In particular, groups nilpotent of class 1 are precisely the abelian groups. We write $G_{\tn{ab}}=\nicefrac{G}{\gamma_2G}$ for the abelianisation of $G$.

Let us introduce the key object used to keep track of the algebraic structure of $\Out(G)$.
\begin{definition}\label{def:AJ_filtration}
The \emph{Andreadakis--Johnson filtration} is the sequence of subgroups $\Torr[i]\leq\Out(G)$ defined by
\begin{equation*}
\Torr[i] \defeq \ker\left(\Out(G)\to\Out\left(\nicefrac{G}{\gamma_{i+1}G}\right)\right).
\end{equation*}
\end{definition}
Clearly, $\{1\}=\Torr[K]\leq\Torr[K-1]\leq\dots\leq\Torr[2]\leq\Torr[1]=\Out(G)$ when $G$ is nilpotent of class $K$. The goal of this section is Lemma \ref{lmm:Johnson_hom}, which gives an explicit description of the successive quotients of the Andreadakis--Johnson filtration.

We start by understanding honest automorphisms.
\begin{lemma}\label{lmm:Johnson_aut}
Suppose that $i\geqslant 1$ and $\Phi\in\Aut(G)$ acts trivially on $\nicefrac{G}{\gamma_{i+1}G}$. Then the assignment $\bar{\delta}\Phi\colon x\mapsto\Phi(x)x^{-1}$ descends to a well-defined $\mathbb{Z}$-module homomorphism $\bar{\delta}\Phi\colon G_{\tn{ab}}\to\nicefrac{\gamma_{i+1}G}{\gamma_{i+2}G}$.

Furthermore, the function
\begin{equation*}
\bar{\delta}\quad\colon\quad\ker\left(\Aut(G)\to\Aut\left(\nicefrac{G}{\gamma_{i+1}G}\right)\right)\quad\to\quad\Hom_{\mathbb{Z}}(G_{\tn{ab}}\to\nicefrac{\gamma_{i+1}G}{\gamma_{i+2}G})
\end{equation*}
is a group homomorphism. Moreover, the kernel of $\bar{\delta}$ consists of precisely those automorphisms which act trivially on $\nicefrac{G}{\gamma_{i+2}G}$.
\end{lemma}
\begin{proof}
We just need to check a number of standard properties.

\smallskip
\textbf{Part 1:} \textit{Each $\bar{\delta}\Phi$ is a $\mathbb{Z}$-module homomorphism $G_\tn{ab}\to\nicefrac{\gamma_{i+1}G}{\gamma_{i+2}G}$.}
\smallskip

We start with $\bar{\delta}\Phi(x)=\Phi(x)x^{-1}\cdot\gamma_{i+2}G$ as a function $G\to\nicefrac{G}{\gamma_{i+2}G}$ between sets. Since $\Phi\curvearrowright\nicefrac{G}{\gamma_{i+1}G}$ by identity, we have $\Phi(x)\in x\cdot\gamma_{i+1}G$, so the image of $\bar{\delta}$ lies in $\nicefrac{\gamma_{i+1}G}{\gamma_{i+2}G}$.

We claim that $\bar{\delta}\Phi$ is actually a group homomorphism $G\to\nicefrac{\gamma_{i+1}G}{\gamma_{i+2}G}$. To confirm this take any $x,y\in G$, and note
\begin{align*}
\bar{\delta}\Phi(xy)&=\Phi(x)\Phi(y)y^{-1}x^{-1}=\\
&=\bar{\delta}\Phi(x)\cdot x\cdot\bar{\delta}\Phi(y)\cdot x^{-1}\in\bar{\delta}\Phi(x)\cdot\bar{\delta}\Phi(y)\cdot\gamma_{i+2}G
\end{align*}
where in the last step we used the fact that $\bar{\delta}\Phi(y)\in\gamma_{i+1}G$, so all of its conjugates are congruent modulo $\gamma_{i+2}G$. This shows that $\bar{\delta}\Phi$ is a group homomorphism.

The target $\nicefrac{\gamma_{i+1}G}{\gamma_{i+2}G}$ is abelian, so $\bar{\delta}\Phi$ must factor through $G_{\tn{ab}}$. This factorisation is a homomorphism between abelian groups.

\smallskip
\textbf{Part 2:} \textit{$\bar{\delta}$ itself is a group homomorphism.}
\smallskip

We just need to verify that $\bar{\delta}$ respects the group composition. Take any $\Phi_1,\Phi_2\in\Aut(G)$ which act trivially on $\nicefrac{G}{\gamma_{i+1}G}$. Then
\begin{align*}
\bar{\delta}(\Phi_1\Phi_2)(x) &=\Phi_1\Phi_2(x)\Phi_2(x)^{-1}\cdot\Phi_2(x)x^{-1}=\\
&=\bar{\delta}\Phi_1\big(\Phi_2(x)\big)\cdot\bar{\delta}(\Phi_2)(x)=\\
&=\bar{\delta}\Phi_1\big(\bar{\delta}\Phi_2(x)\big) \cdot \bar{\delta}\Phi_1(x) \cdot \bar{\delta}\Phi_2(x).
\end{align*}
But we know that $\bar{\delta}\Phi_2(x)$ is zero in $G_\tn{ab}$, so it is killed by $\bar{\delta}\Phi_1$, and the first factor vanishes. Therefore $\bar{\delta}(\Phi_1\Phi_2)=\bar{\delta}\Phi_1+\bar{\delta}\Phi_2$, so $\bar{\delta}$ is a group homomorphism from $\ker\left(\Aut(G)\to\Aut\left(\nicefrac{G}{\gamma_{i+1}G}\right)\right)$ to $\Hom_{\mathbb{Z}}(G_{\tn{ab}}\to\nicefrac{\gamma_{i+1}G}{\gamma_{i+2}G})$.

\smallskip
\textbf{Part 3:} \textit{Kernel of $\bar{\delta}$.}
\smallskip

Finally $\bar{\delta}\Phi=0$ is equivalent to having $\Phi(x)\in x\cdot\gamma_{i+2}G$ for every $x\in G$, which precisely means that $\Phi\curvearrowright\nicefrac{G}{\gamma_{i+2}G}$ by identity.
\end{proof}

We are ready to describe the key tool to understand the algebraic structure of automorphisms of nilpotent groups -- the Johnson homomorphisms. Denote
\begin{equation}\label{eq:H}
H_i\defeq\left\{x\in G\ \big|\ [x,y]\in\gamma_{i+1}G \tn{ for all } y\in G\right\}.
\end{equation}
Note that $\gamma_iG\leq H_i$. Also, with the notation from Lemma \ref{lmm:Johnson_aut}, each $x\in H_i$ induces a $\mathbb{Z}$-module homomorphism $\bar{\delta}(\ad_x)\colon G_{\tn{ab}}\to\nicefrac{\gamma_{i+1}G}{\gamma_{i+2}G}$ (which is nothing other than the group commutator $[x,\bullet]$).
\begin{lemma}[Johnson homomorphism]\label{lmm:Johnson_hom}
Let $\bar{\delta}$ be the homomorphism from Lemma \ref{lmm:Johnson_aut}. For $\varphi\in\Torr[i]$, pick a representative $\Phi\in\Aut(G)$ which acts trivially on $\nicefrac{G}{\gamma_{i+1}G}$. Then the assignment $\delta\colon\varphi\mapsto\bar{\delta}\Phi$ is a well-defined $\mathbb{Z}$-module homomorphism
\begin{equation*}
\delta\colon\Torr[i]\to D_i \qquad\text{where}\qquad D_i\defeq\frac{\Hom_{\mathbb{Z}}(G_{\tn{ab}}\to\nicefrac{\gamma_{i+1}G}{\gamma_{i+2}G})}{\left\{\bar{\delta}(\ad_x)\ \big|\ x\in H_i\right\}}.
\end{equation*}
We call $\delta$ the \emph{Johnson homomorphism}. Moreover, $\ker\delta=\Torr[i+1]$.
\end{lemma}

An immediate corollary of Lemma \ref{lmm:Johnson_hom} is that $\nicefrac{\Torr[i]}{\Torr[i+1]}$ is abelian.
\begin{remark}
Although Lemma \ref{lmm:Johnson_hom} contains all of the algebra needed for our purposes, there is more structure to Johnson homomorphisms. To any finitely generated residually (torsion-free nilpotent) group one can associate a graded Lie algebra, by taking the direct sum of the successive quotients, and inducing Lie bracket from the group commutator. Then the homomorphisms $\bar{\delta}$ from Lemma \ref{lmm:Johnson_aut} produce graded derivations of the Lie algebra, and the $\mathbb{Z}$-linear maps that we use are recovered by restricting the derivations to the grade 1 subspace. Moreover, $\bar{\delta}$ sends group-theoretic commutator in the autmorphism group to the Lie-algebra-theoretic commutator of derivations. The normal subgroup of inner automorphisms is sent to the Lie ideal of inner derivations. For a detailed exposition see e.g. \cite[Section 2]{Andreadakis-Johnson_filtration_on_AutFn_CohenHeapPettet11}.
\end{remark}

\begin{proof}[Proof of Lemma \ref{lmm:Johnson_hom}] First, note that $\bar{\delta}\circ\ad$ is a group homomorphism from $H_i$ to $\Hom_{\mathbb{Z}}(G_\tn{ab}\to\nicefrac{\gamma_{i+1}G}{\gamma_{i+2}G})$. Therefore $\left\{\bar{\delta}(\ad_x)\ \big|\ x\in H_i\right\}$ is a sub-$\mathbb{Z}$-module, so $D_i$ is well-defined.

Second, we check well-definedness. For any $\phi\in\Torr[i]$, a representative $\Phi\in\ker(\Aut(G)\to\Aut(\nicefrac{G}{\gamma_{i+1}G}))$ exists, and is unique up to conjugation by an element of $H_i$. But this ambiguity vanishes in $D_i$ by definition. Therefore $\delta\colon\Torr[i]\to D_i$ is well-defined as a function between sets.

Third, we check that $\delta$ is a group homomorphism. This is a consequence of $\bar{\delta}$ being a group homomorphism, as asserted by Lemma \ref{lmm:Johnson_aut}.

Finally we characterise the kernel. By definition $\delta\varphi=0$ iff there exists $x\in H_i$ such that $\bar{\delta}\Phi=\bar{\delta}(\ad_x)$ for some representative $\Phi$ of $\varphi$ which acts trivially on $\nicefrac{G}{\gamma_{i+1}G}$. This means that $\Phi\circ\ad_x^{-1}\in\ker\bar{\delta}$. By Lemma \ref{lmm:Johnson_aut}, this happens precisely when $\Phi\circ\ad_x^{-1}$ acts trivially on $\nicefrac{G}{\gamma_{i+2}G}$. And this is equivalent to $\varphi\in\Torr[i+1]$.
\end{proof}

\section{CSP for automorphisms of nilpotent groups}\label{sec:csp_nilpotent}

The goal of this section is to prove Theorem \ref{thm:csp_nilpotent}. We start by proving in Subsection \ref{ssec:auxiliary} two Lemmata \ref{lmm:seeing_bottom} and \ref{lmm:carve_centre}, which are technical tools for producing special quotients of nilpotent groups. We complete the proof of Theorem \ref{thm:csp_nilpotent} at the end of Subsection \ref{ssec:pf_csp_nilpotent}, after proving its special case Proposition \ref{prop:csp_lastterm}.

The overall proof approach is straightforward: to induct on the length of central series with the help of Lemma \ref{lmm:csp_ses}. This is perhaps closest to the strategy outlined in \cite[Section 4]{csp_fg_nilp_gps_BenEzraLubotzky22}; however, thanks to Lemmata \ref{lmm:seeing_bottom} and \ref{lmm:carve_centre}, we will be able to complete the inductive step for arbitrary (including non-free) nilpotent groups.

\subsection{Auxiliary results}\label{ssec:auxiliary}

In this subsection we bundle the technical results towards Proposition \ref{prop:csp_lastterm} and Theorem \ref{thm:csp_nilpotent}. The first one is a certain strenghtening of residual finiteness of nilpotent groups.
\begin{lemma}\label{lmm:seeing_bottom}
Let $G$ be a finitely generated nilpotent group of class $K$, and take any $M\in\mathbb{N}_+$. There exists a finite-index subgroup $L\leq G$ such that $L\cap\gamma_KG=M\cdot\gamma_KG$.
\end{lemma}
By $M\cdot\gamma_KG$ we mean the set of all $M$-th multiples in the abelian group $\gamma_KG$. An alternative way to phrase Lemma \ref{lmm:seeing_bottom} is to say that the modulo-$M$ map $\gamma_KG\to(\gamma_KG)\otimes\left(\tfrac{\mathbb{Z}}{M\mathbb{Z}}\right)$ factors through $\gamma_KQ$ for some finite quotient $G\to Q$.
\begin{proof}[Proof of Lemma \ref{lmm:seeing_bottom}]
Refine the lower central series of $G$ to a sequence
\begin{equation*}
G=F_1 \triangleright F_2 \triangleright \dots \triangleright F_{K'-1} \triangleright F_{K'}=\gamma_KG
\end{equation*}
such that the successive quotients $\nicefrac{F_i}{F_{i+1}}$ for $1\leqslant i\leqslant K'-1$ are abelian of rank one. We recursively construct a sequence of finite-index subgroups $L_i\leq F_i$ such that $L_i\cap \gamma_KG=M\cdot\gamma_KG$

We start by setting $L_{K'}=M\cdot\gamma_KG$. Since $G$ is finitely generated, then so is $\gamma_KG$, and hence $L_{K'}$ is finite-index in $F_{K'}=\gamma_KG$.

Suppose we have $L_{i+1}$. If $\nicefrac{F_i}{F_{i+1}}$ is finite cyclic, then set $L_i=L_{i+1}$. Otherwise the group $F_i$ splits as $F_{i+1}\rtimes\langle t\rangle$ for some $t\in F_i$ of infinite order. Since $L_{i+1}$ is a finite-index subgroup of a finitely generated $F_{i+1}$, its orbit under the conjugation action of $t$ is finite. Hence there exists $m\in\mathbb{N}_+$ such that conjugation by $t^m$ preserves $L_{i+1}$ setwise. Define $L_i=L_{i+1}\cdot\langle t^m\rangle$. This is a subgroup of $F_i$, its index is finite, and $L_i\cap F_{i+1}=L_{i+1}$, which implies that $L_i\cap\gamma_KG=L_{i+1}\cap\gamma_KG=M\cdot\gamma_KG$.

Finally, $L_1$ satisfies the premises of Lemma \ref{lmm:seeing_bottom}.
\end{proof}

The second technical result produces finite quotients with a small centre.
\begin{lemma}\label{lmm:carve_centre}
Suppose $L\leq G$ is a finite index subgroup of a finitely generated nilpotent group $G$. There exists a finite quotient $r\colon G\to R$ such that $r^{-1}(Z(R))\subseteq Z(G)\cdot L$.
\end{lemma}
\begin{proof}
Denote the nilpotency class of $G$ by $K$, and let $H_i\leq G$ be as in equation (\ref{eq:H}). Note
\begin{equation*}
Z(G)=H_K\leq H_{K-1}\leq\dots\leq H_2\leq H_1=G.
\end{equation*}

We will recursively construct a sequence of quotients $q_i\colon G\to Q_i$ satisfying the conditions
\begin{equation}\label{eq:carving_ind_inv}
H_i\cap q_i^{-1}(Z(Q_i))\subseteq Z(G)\cdot L.
\end{equation}
We start with $(q_K,Q_K)$ trivial, since $H_K=Z(G)$ ensures (\ref{eq:carving_ind_inv}) at index $i=K$.

\smallskip
\textbf{Step 1:} \textit{Construction of $(q_i,Q_i)$ from $(q_{i+1},Q_{i+1})$.}
\smallskip

Suppose $q_{i+1}\colon G\to Q_{i+1}$ satisfies condition (\ref{eq:carving_ind_inv}). Any element $x\in H_i$ induces an inner automorphism $\ad_x$ which acts on $\nicefrac{G}{\gamma_{i+1}G}$ by identity. Lemma \ref{lmm:Johnson_aut} provides us with an exact sequence
\begin{equation*}\begin{tikzcd}
H_{i+1}\arrow[r]& H_i\arrow[r,"\bar{\delta}\circ\ad"]& \Hom_{\mathbb{Z}}\left(G_{\tn{ab}}\to\nicefrac{\gamma_{i+1}G}{\gamma_{i+2}G}\right)
\end{tikzcd}\end{equation*}
where explicitly $\bar{\delta}(\ad_x)\colon y\cdot\gamma_2G\mapsto [x,y]\cdot\gamma_{i+2}G$ for $x\in H_i$ and $y\cdot\gamma_2G\in G_{\tn{ab}}$.

Set $M_1=|Q_{i+1}|\cdot [G:L]!$, and pick $M_2\in\mathbb{N}_+$ such that
\begin{equation}\label{eq:M_2}
\left(\im\bar{\delta}\circ\ad\right)\cap\big(M_2\cdot\Hom_{\mathbb{Z}}\left(G_{\tn{ab}}\to\nicefrac{\gamma_{i+1}G}{\gamma_{i+2}G}\right)\big)\ \subseteq\ M_1\cdot\left(\im\bar{\delta}\circ\ad\right).
\end{equation}
The factorial in the definition of $M_1$ will be necessary later, because technically we did not assume that $L\triangleleft G$. Also, note that we cannot simply take $M_2=M_1$, because there might be elements of $\Hom_{\mathbb{Z}}(G_{\tn{ab}}\to\nicefrac{\gamma_{i+1}G}{\gamma_{i+2}G})$ which land in the submodule $\im(\bar{\delta}\circ\ad)$ only after taking a proper multiple.

By Lemma \ref{lmm:seeing_bottom}, there exists a finite quotient $r\colon\nicefrac{G}{\gamma_{i+2}G}\onto R$ which fits in the commutative diagram
\begin{equation*}\begin{tikzcd}
& \nicefrac{\gamma_{i+1}G}{\gamma_{i+2}G}\arrow[d,"r"]\arrow[r,hook]\arrow[dl,"\mu"']& \nicefrac{G}{\gamma_{i+2}G}\arrow[d,"r"]\\
\left(\nicefrac{\gamma_{i+1}G}{\gamma_{i+2}G}\right)\otimes\left(\nicefrac{\mathbb{Z}}{M_2\mathbb{Z}}\right)& \gamma_{i+1}R\arrow[hook,r]\arrow[l]& R
\end{tikzcd}\end{equation*}
where $\mu$ is the natural modulo-$M_2$ map of $\mathbb{Z}$-modules. Note that $\gamma_{i+1}R$ is central in $R$. Define $Q_i$ to be the image of $G$ in $Q_{i+1}\times R$ under $q_{i+1}\times r$, and $q_i\colon G\onto Q_i$ to be the corresponding factorisation.

\smallskip
\textbf{Step 2:} \textit{The quotient $(q_i,Q_i)$ satisfies the condition (\ref{eq:carving_ind_inv}).}
\smallskip

We now have $(q_i,Q_i)$, and we claim (\ref{eq:carving_ind_inv}) holds at index $i$. To verify this, take any $x\in H_i$ such that $q_i(x)\in Z(Q_i)$. Since $x\in H_i$, it defines a $\mathbb{Z}$-linear map
\begin{equation*}
\bar{\delta}(\ad_x)\colon G_{\tn{ab}}\to\nicefrac{\gamma_{i+1}G}{\gamma_{i+2}G}
\end{equation*}
in accordance with Lemma \ref{lmm:Johnson_aut}. But the image of $x$ in $R$ is central, so for every $w\in G$, we have that $\bar{\delta}(\ad_x)(w\cdot\gamma_2G)=[x,w]\cdot\gamma_{i+2}G$ is in the kernel of $r$. This means that any vector in the image of $\bar{\delta}(\ad_x)$ must be an $M_2$-th multiple in the $\mathbb{Z}$-module $\nicefrac{\gamma_{i+1}G}{\gamma_{i+2}G}$. Therefore
\begin{equation*}
\bar{\delta}(\ad_x)\in M_2\cdot\Hom_{\mathbb{Z}}\left(G_{\tn{ab}}\to\nicefrac{\gamma_{i+1}G}{\gamma_{i+2}G}\right).
\end{equation*}
By the choice (\ref{eq:M_2}) of $M_2$, this means that $\bar{\delta}(\ad_x)$ is an $M_1$-th multiple of an element of $\im(\bar{\delta}\circ\ad)$. Write this as $\bar{\delta}(\ad_x)=M_1\cdot\bar{\delta}(\ad_y)$ for some $y\in H_i$. The kernel of $\bar{\delta}\circ\ad$ is precisely $H_{i+1}$, so
\begin{equation*}
xy^{-M_1}=z
\end{equation*}
for some $z\in H_{i+1}$.

Now we examine the consequences of this equation in $Q_i$. From Lagrange's theorem we know that $q_{i+1}$ annihilates $M_1$-th powers, so $q_{i+1}(y^{M_1})$ is trivial in $Q_{i+1}$. Therefore $q_{i+1}(z)=q_{i+1}(x)$. By assumption $q_i(x)$ is central in $Q_i$, so we must also have $q_{i+1}(x)\in Z(Q_{i+1})$ in its quotient $Q_{i+1}$. Since $z\in H_{i+1}$ and $q_{i+1}(z)\in Z(Q_{i+1})$, from condition (\ref{eq:carving_ind_inv}) at the index $i+1$ we conclude $z\in Z(G)\cdot L$. But $M_1$-th powers are forced to lie in $L$ (because $[G:L]!$ divides $M_1$), so $y^{M_1}\in L$. Therefore $x\in Z(G)\cdot L$, confirming the condition (\ref{eq:carving_ind_inv}) at the index $i$.

\bigskip
Finally, $(q_1,Q_1)$ satisfies (\ref{eq:carving_ind_inv}) at the index $i=1$, which means that it verifies the premises of Lemma \ref{lmm:carve_centre}.
\end{proof}

\subsection{Proof of Theorem \ref{thm:csp_nilpotent}}\label{ssec:pf_csp_nilpotent}

As a step towards Theorem \ref{thm:csp_nilpotent} we prove Proposition \ref{prop:csp_lastterm}, which gives us CSP of arbitary subgroups of the last term of the Andreadakis--Johnson filtration. It follows from Lemmata \ref{lmm:seeing_bottom} and \ref{lmm:carve_centre}.
\begin{proposition}\label{prop:csp_lastterm}
Let $G$ be a finitely generated nilpotent group of class $K\geqslant 2$, and $\Lambda\leq\Torr[K]$. Then the inclusion $\Lambda\hookrightarrow\Out(G)$ has CSP.
\end{proposition}
Note that by Lemma \ref{lmm:Johnson_hom} the subgroup $\Torr[K]$ is abelian, so $\Lambda$ can only be a sub-$\mathbb{Z}$-module.
\begin{proof}[Proof of Proposition \ref{prop:csp_lastterm}] Lemma \ref{lmm:Johnson_hom} provides us with an injective group homomorphism $\delta\colon\Torr[K]\to D_K$. Moreover, $D_K$ is a finitely generated $\mathbb{Z}$-module. Therefore, for every finite-index $\Gamma\leq\Lambda$ there exists $M\in\mathbb{N}_+$ such that $\delta(\Lambda)\cap M\cdot D_K\subseteq\delta(\Gamma)$. To finish the proof it suffices to show that for every $M\in\mathbb{N}_+$ there exists a finite characteristic quotient $r\colon G\onto R$ such that
\begin{equation}\label{eq:cong_R_in_M}
\ker\left(r_\star\colon\Torr[K]\to\Out(R)\right) \subseteq \delta^{-1}(M\cdot D_K).
\end{equation}

\smallskip
\textbf{Step 1:} \textit{Construction of the quotient $(r,R)$.}
\smallskip

Lemma \ref{lmm:seeing_bottom} says that there exists a finite quotient $p\colon G\onto P$ such that $\gamma_KG\cap\ker p\subseteq M\cdot\gamma_KG$.

The map $p$ induces a finite quotient $\nicefrac{G}{\gamma_KG}\onto\nicefrac{P}{\gamma_KP}$. By Lemma \ref{lmm:carve_centre}, applied to the group $\nicefrac{G}{\gamma_KG}$ with subgroup $(\ker p)\cdot\gamma_KG$, there exists a finite quotient $q\colon\nicefrac{G}{\gamma_KG}\onto Q$ such that
\begin{equation*}
q^{-1}(Z(Q))\subseteq Z(\nicefrac{G}{\gamma_KG})\cdot\left(\nicefrac{\ker p}{\gamma_KG}\right)
\end{equation*}
holds in $\nicefrac{G}{\gamma_KG}$. We can consider $(q,Q)$ as a quotient of $G$, and then the above condition means that
\begin{equation}\label{eq:centre_Q}
q^{-1}(Z(Q))\subseteq H_K\cdot(\ker p)
\end{equation}
holds in $G$, where $H_K$ is defined in equation (\ref{eq:H}).

Let $(r,R)$ be any characteristic quotient such that the product $p\times q\colon G\to P\times Q$ factors through $r\colon G\onto R$.

\smallskip
\textbf{Step 2:} \textit{Representatives for elements of $\ker r_\star$.}
\smallskip

In this step we prove that any $\varphi\in\ker\left(r_\star\colon\Torr[K]\to\Out(R)\right)$ has a representative $\Phi\in\Aut(G)$ which acts trivially both on $\nicefrac{G}{\gamma_KG}$ and on the quotient $P$.

Start with a representative $\Phi'\in\Aut(G)$ which acts as the identity on $\nicefrac{G}{\gamma_KG}$. It exists because $\varphi\in\Torr[K]$. Since $\Phi'$ becomes inner in $R$, there exists $w\in G$ such that for all $x\in G$ the elements $\Phi'(x)$ and $wxw^{-1}$ differ by an element of $\ker r$.

Let us we examine $\Phi'\curvearrowright Q$. On the one hand, $\Phi'$ acts trivially on $\nicefrac{G}{\gamma_KG}$, so it must also act trivially on $Q$. On the other hand, the action $\Phi'\curvearrowright Q$ agrees with conjugation by $q(w)$. Hence we must have $q(w)\in Z(Q)$. The defining property (\ref{eq:centre_Q}) of $(q,Q)$ now means precisely that we can write $w=uv$ for some $u\in H_K$ and $v\in\ker p$.

Now we adjust the representative of $\varphi$ by choosing the automorphism $\Phi=\ad_u^{-1}\circ\Phi'$. We know that $u\in H_K$, so $\ad_u$ becomes the identity over $\nicefrac{G}{\gamma_KG}$, and we still have $\Phi\curvearrowright\nicefrac{G}{\gamma_KG}$ trivially. The induced action of $\Phi$ on the quotient $P$ agrees with conjugation by $p(v)$; but we know that $u\in\ker p$, so this action is trivial too.

\smallskip
\textbf{Step 3:} \textit{Principal congruence subgroup induced by $(r,R)$ lies in $\delta^{-1}(M\cdot D_K)$.}
\smallskip

We are ready to prove equation (\ref{eq:cong_R_in_M}). Consider any $\varphi\in\ker r_\star$. From the previous step we know that it has a representative $\Phi\in\Aut(G)$ such that $\Phi(x)$ belongs to $x\cdot\gamma_KG$ and $x\cdot\ker p$ for every $x\in G$.

From the defining property of $(p,P)$ we know that $\Phi(x)x^{-1}\in M\cdot\gamma_KG$ for all $x\in G$. This means that the $\mathbb{Z}$-module homomorphism $G_\textnormal{ab}\to\gamma_KG$ defined by $x\mapsto\Phi(x)x^{-1}$ only attains $M$-th multiples as its values. Therefore this homomorphism is itself an $M$-th multiple of some other $\mathbb{Z}$-module homomorphism. We have found a representative of $\delta\phi$ in $\Hom_\mathbb{Z}(G_\textnormal{ab}\to\gamma_KG)$ which is an $M$-th multiple, so $\delta\phi\in M\cdot D_K$, confirming equation (\ref{eq:cong_R_in_M}).
\end{proof}

We are ready to complete the proof of Theorem \ref{thm:csp_nilpotent}. We will inductively apply Lemma \ref{lmm:csp_ses}, and use Proposition \ref{prop:csp_lastterm} to do the heavy lifting in the inductive step. Let us recall the statement.
\begin{theorem*}[restatement of \ref{thm:csp_nilpotent}]
Suppose $G$ is a finitely generated nilpotent group, and $\Lambda\leq\Out(G)$. If the linear action on abelianisation $\Lambda\to\Out(G_\textnormal{ab})$ has CSP, then $\Lambda\to\Out(G)$ has CSP.
\end{theorem*}
\begin{proof}[Proof of Theorem \ref{thm:csp_nilpotent}]
We induct on the nilpotence class $K$ of $G$. Base $K=1$ corresponds to abelian $G$, and the conclusion is the same as the assumption.

It remains to do the inductive step. Let $G$ be nilpotent of class $K\geqslant 2$, and $\Lambda\leq\Out(G)$ be a subgroup whose image in $\Gl(G_{\tn{ab}})$ has CSP. Define $\Delta=\Lambda\cap\Torr[K]$. From Proposition \ref{prop:csp_lastterm} we know that $\Delta\to\Out(G)$ has CSP. By Lemma \ref{lmm:csp_ses}, it is enough to show that every finite-index subgroup of $\Lambda$ containing $\Delta$ is congruence.

Observe that the kernel of the composition $\Lambda\to\Out(G)\to\Out\left(\nicefrac{G}{\gamma_KG}\right)$ is precisely $\Delta$. Moreover, the abelianisations of $\nicefrac{G}{\gamma_KG}$ and $G$ are naturally identified, and the action of $\Lambda$ on them factors through $\nicefrac{\Lambda}{\Delta}$.

Consider a finite-index $\Gamma\leq\Lambda$ containing $\Delta$. From the induction hypothesis, we know that $\nicefrac{\Lambda}{\Delta}\to\Out\left(\nicefrac{G}{\gamma_KG}\right)$ has CSP. Therefore $\nicefrac{\Gamma}{\Delta}$ is a congruence subgroup of $\nicefrac{\Lambda}{\Delta}$ with respect to $\nicefrac{\Lambda}{\Delta}\to\Out\left(\nicefrac{G}{\gamma_KG}\right)$. Therefore $\Gamma$ is a congruence subgroup of $\Lambda$ with respect to $\Lambda\to\Out\left(\nicefrac{G}{\gamma_KG}\right)$. But then it must also be a congruence subgroup with respect to $\Lambda\to\Out(G)$.
\end{proof}

\section{Proof of Theorem \ref{thm:csp_subsurface}}\label{sec:pf_csp_subsurface}

The aim of this section is to prove Theorem \ref{thm:csp_subsurface}. It is organised as follows. As prerequisites we investigate what happens at the boundaries of subsurfaces by proving Lemma \ref{lmm:csp_multitwist} in Subsection \ref{ssec:bdry_twists}. Next, in Subsection \ref{ssec:virt_retraction} we introduce Lemma \ref{lmm:virt_retraction}, which is a tool to extend quotients of fundamental group of subsurfaces to the fundamental group of the whole surface. The proof of Theorem \ref{thm:csp_subsurface} is completed at the end of Subsection \ref{ssec:pf_csp_subsurface}, after the main technical argument in Proposition \ref{prop:subsurface_containment}.

Before proceeding we recall the setup of Theorem \ref{thm:csp_subsurface}, with explicit notation. We have a surface $\Sigma$ of finite type, possibly punctured but without boundary, and a collection of disjoint simple closed curves $B=\{\beta_j\}_j$ which cuts $\Sigma$ into subsurfaces $\Sigma_i$ with boundary. Denote the inclusion maps by $\iota_i\colon\Sigma_i\hookrightarrow\Sigma$. Each subsurface $\Sigma_i$ comes with a subgroup $\Lambda_i\leq\Mcg(\Sigma_i)$, which produces $\iota_{i,\star}(\Lambda_i)\leq\Mcg(\Sigma)$. We are interested in the subgroup $\Lambda=\left\langle\iota_{i,\star}(\Lambda_i)\right\rangle$ of $\Mcg(\Sigma)$ generated by them. Algebraically, $\Lambda$ is isomorphic to the quotient of the direct product $\prod_i\Lambda_i$ by the elements of the form $T_{j,+}T_{j,-}^{-1}$, where $T_{j,\pm}$ are Dehn twists about the boundary $\beta_j$ in the subsurfaces on two sides of $\beta_j$. Note that boundary twists $T_{j,\pm}$ are central in Mapping Class Groups of the corresponding subsurfaces.

Let us explain the assumption of $\Lambda_i\to\Out(\pi_1\Sigma_i)$ having CSP in different words, which some readers may find clearer. According to Definition \ref{def:csp}, the meaning of this statement is that every finite-index $\Gamma\leq\Lambda_i$, such that $(\Lambda_i\setminus\Gamma)$ contains no boundary multitwists, is congruence. A way to state this without the annoying kernel is to consider the surface $\overline{\Sigma}_i$ obtained from $\Sigma_i$ by capping each boundary component with a once-punctured disc, and to look at the image $\overline{\Lambda}_i$ of $\Lambda_i$ in $\Mcg(\overline{\Sigma}_i)$. Then the assumption is equivalent to saying that every finite-index subgroup of $\overline{\Lambda}_i$ is congruence with respect to $\overline{\Lambda}_i\to\Out(\pi_1\overline{\Sigma}_i)$.

\subsection{Boundary twists}\label{ssec:bdry_twists}

Subsurface subgroups are amalgamated along the boundary twists. To take care of them we will use the following lemma.
\begin{lemma}\label{lmm:csp_multitwist}
Let $\mathcal{B}$ be a collection of disjoint simple closed curves on the surface $\Sigma$, and $\Lambda$ be a subgroup of $\Mcg(\Sigma)$ generated by an arbitrary collection of multitwists about curves from $\mathcal{B}$. Then $\Lambda\to\Out(\pi_1\Sigma)$ has the CSP.
\end{lemma}
\begin{proof}
Without loss of generality we can assume that elements of $\mathcal{B}$ are pairwise non-homotopic and essential.

Consider the linearised action $\Lambda\to\Out(\tn{H}_1(\Sigma;\mathbb{Z}))$. We claim that it has the CSP. To see this, observe that the image of $\Lambda$ in $\tn{Sp}(\tn{H}_1(\Sigma;\mathbb{Z}))$ is an abelian subgroup generated by a collection of products of commuting transpositions. There exists a basis for $\tn{H}_1(\Sigma;\mathbb{Z})$ in which these transpositions are unitriangular, and matrix multiplication becomes addition of off-diagonal entries. Then, every finite-index subgroup of the image of $\Lambda$ contains the subgroup of matrices congruent to identity modulo some $M\in\mathbb{N}_+$. But this is the principal congruence subgroup induced by the quotient $\tn{H}_1(\Sigma;\tfrac{\mathbb{Z}}{M\mathbb{Z}})$.

We have argued that $\Lambda\curvearrowright\tn{H}_1(\Sigma_1;\mathbb{Z})$ has the CSP. Write $\gamma_3\pi_1\Sigma$ for the third term in the lower central series of $\pi_1\Sigma$, and $\Torr[2]$ for the second term in the Johnson filtration (as in Section \ref{sec:Johnson_homs} and Definition \ref{def:AJ_filtration}). Theorem \ref{thm:csp_nilpotent} tells us that the homomorphism $\nicefrac{\Lambda}{\Torr[2]}\to\Out\left(\nicefrac{\pi_1\Sigma}{\gamma_3\pi_1\Sigma}\right)$ has the CSP.

We claim that the map $\Lambda\to\nicefrac{\Lambda}{\Torr[2]}$ is injective. Assuming this statement, we can deduce that $\Lambda\to\Out\left(\nicefrac{\pi_1\Sigma}{\gamma_3\pi_1\Sigma}\right)$ has the CSP. Since every quotient of $\nicefrac{\pi_1\Sigma}{\gamma_3\pi_1\Sigma}$ is also naturally a quotient of $\pi_1\Sigma$, this implies that $\Lambda\to\Out(\pi_1\Sigma)$ has the CSP and we are done.

It only remains to justify the claim that $\Lambda\cap\Torr[2]$ is trivial. Let $\Delta_N$ and $\Delta_S$ be the subgroups of $\Mcg(\Sigma)$ generated by Dehn twists about non-separating and separating loops in $\mathcal{B}$ respectively. Since $\Lambda\leq\Delta_N\Delta_S$, it is enough to prove that $\Delta_N\Delta_S\cap\Torr[2]$ is trivial. Note that the symplectic representation (which algebraically is the quotient by $\Torr$) is faithful on $\Delta_N$ and annihilates $\Delta_S$, which means that $\Delta_N\Delta_S\cap\Torr[2]\subseteq\Delta_N\Delta_S\cap\Torr\subseteq\Delta_S$. Therefore, it is enough to show that $\Delta_S\cap\Torr[2]$ is trivial. This can be seen by noting that for every separating curve $\beta\in\mathcal{B}$, there exists a non-separating simple closed curve $\alpha$ on $\Sigma$ which has geometric intersection number 2 with $\beta$ and 0 with any other separating curve from $\mathcal{B}$. Then we examine the action of $\Delta_S$ on the class of $\alpha$ in $\nicefrac{\pi_1\Sigma}{\gamma_3\pi_1\Sigma}$, and conclude that it is preserved by every $T_{\beta'}$ for separating $\beta'\in\mathcal{B}\setminus\{\beta\}$, but moved by any nonzero power of $T_\beta$.
\end{proof}

\subsection{Virtual retractions}\label{ssec:virt_retraction}

In this subsection we introduce Lemma \ref{lmm:virt_retraction} that allows us to convert quotients of subsurfaces into quotients of the whole fundamental group.

\begin{lemma}[Virtual retraction]\label{lmm:virt_retraction}
Let $d\geqslant 2$, and $\iota\colon(\Sigma',x)\hookrightarrow(\Sigma,x)$ be an inclusion of pointed connected surfaces. Assume that no component of the complement $\Sigma\setminus\iota(\Sigma')$ is a disc, and not all of them are once-punctured discs. Then $\iota$ factors as
\begin{equation*}\begin{tikzcd}
(\Sigma',x)\arrow[r,"\hat{\iota}"]\arrow[rr,bend right=20,"\iota"] & (\widehat{\Sigma},\hat{x})\arrow[r,"F"] & (\Sigma,x)
\end{tikzcd}\end{equation*}
where $F$ is a Galois covering with deck group $\tfrac{\mathbb{Z}}{d\mathbb{Z}}$, and there exists a continuous retraction $\rho\colon(\widehat{\Sigma},\hat{x})\to(\Sigma',x)$ such that $\rho\circ\hat{\iota}=\textnormal{id}_{\Sigma'}$.
\end{lemma}

\begin{figure}[p]
\begin{minipage}{.5\textwidth}\begin{center}
\begin{tikzpicture}[scale=.94]
\draw[dashed,fill=black!10!white] (0,2) to[bend left] (0,1) to[bend left] (0,2);
\draw[dash pattern={on 2pt off 2pt},fill=black!10!white] (0,0) to[bend left] (0,-.4) to[bend left] (0,0) (0,-.8) to[bend left] (0,-1.2) to[bend left] (0,-.8) (0,-1.6) to[bend left] (0,-2) to[bend left] (0,-1.6);
\draw[dashed,fill=black!10!white] (0,-4) to[bend left] (0,-3) to[bend left=20] (0,-4);
\draw[dashed,fill=black!10!white] (-2,-4) to[bend left] (-1,-4) to[bend left] (-2,-4);
\draw[fill=lightgray] (0,0) to[bend left=50] (0,1) to[bend left] (0,2) to[bend right] (-2,-4) to[bend left] (-1,-4)--(0,-4) to[bend left] (0,-3) to[bend left=50] (0,-2) to[bend left] (0,-1.6) ..controls (-.3,-1.6) and (-.3,-1.2).. (0,-1.2) to[bend left] (0,-.8) ..controls (-.3,-.8) and (-.3,-.4).. (0,-.4) to[bend left] (0,0);
\draw[fill=white] (-1,-1) to[bend left] (-1,-3) to[bend left] (-1,-1);\draw (-1,-1)--+(.15,.2) (-1,-3)--+(.15,-.2);
\draw[fill] (-.5,0) circle (.05) node[anchor=south] {$x$};
\draw (-1,-4) arc (0:-180:.5);
\draw (0,1) .. controls (2,0) and (2,3) .. (0,2);
\draw (.5,1.5) to[bend left] (1,1.5) (.4,1.6) to[bend right=60] (1.1,1.6);
\draw (0,-.4) arc (90:-90:.2) (0,-1.2) arc (90:-90:.2) (0,-2) arc (-90:90:1);
\draw[red] (.94,1.45) to[bend left] (1.29,1) (.18,-.5) to[bend right] (.7,-.3);\draw[red,dashed] (.94,1.45) to[bend right] (1.29,1) (.18,-.5) to[bend left] (.7,-.3);
\draw[fill] (-1.5,-4.35) circle (.05);
\draw (0,-4) arc (-90:90:.5);
\draw[red] (.2,-3.2) to[bend left=80] (.2,-3.8);
\draw[fill] (.2,-3.8) circle (.05) (.2,-3.2) circle (.05);
\node at (-1.8,.8) {$\Sigma'$};\node at (1,2.5) {$\Sigma^{G}$};\node at (1.3,-1.5) {$\Sigma^B$};\node at (-.8,-4.4) {$\Sigma^D$};\node at (.5,-2.7) {$\Sigma^P$};\node[red] at (1.5,.9) {$\beta^G$};\node[red] at (1,-.1) {$\beta^B$};\node[red] at (.8,-3.5) {$\beta^P$};
\end{tikzpicture}
\end{center}\end{minipage}
\hfill
\begin{minipage}{.4\textwidth}\begin{center}
\begin{tikzpicture}[scale=.94]
\draw[dashed,fill=black!10!white] (0,2) to[bend left] (0,1) to[bend left] (0,2);
\draw[fill=black!10!white] (-2,-4) to[bend left] (-1,-4) to[bend left] (-2,-4);
\draw[dashed,fill=black!10!white] (0,-3) to[bend left] (0,-4) to[bend left] (0,-3);
\draw[dash pattern={on 2pt off 2pt},fill=black!10!white] (0,0) to[bend left] (0,-.4) to[bend left] (0,0) (0,-.8) to[bend left] (0,-1.2) to[bend left] (0,-.8) (0,-1.6) to[bend left] (0,-2) to[bend left] (0,-1.6);
\draw[fill=lightgray] (0,0) to[bend left=50] (0,1) to[bend left] (0,2) to[bend right] (-2,-4) to[bend left] (-1,-4)--(0,-4) to[bend left] (0,-3) to[bend left=50] (0,-2) to[bend left] (0,-1.6) ..controls (-.3,-1.6) and (-.3,-1.2).. (0,-1.2) to[bend left] (0,-.8) ..controls (-.3,-.8) and (-.3,-.4).. (0,-.4) to[bend left] (0,0);
\draw[fill=white] (-1,-1) to[bend left] (-1,-3) to[bend left] (-1,-1);\draw (-1,-1)--+(.15,.2) (-1,-3)--+(.15,-.2);
\draw[fill] (-.5,0) circle (.05) node[anchor=south] {$x$};
\draw (0,0) to[bend right=50] (0,1) (0,-2) to[bend left=50] (0,-3) (0,-4)--(1,-4) to[bend left] (2,-4) (2,-4) edge[out=90,in=-90] (1.5,.5) (1,1) edge[out=180,in=-30] (0,2) (0,-1.6) ..controls (.3,-1.6) and (.3,-1.2).. (0,-1.2) (0,-.8) ..controls (.3,-.8) and (.3,-.4).. (0,-.4);
\draw (1,-4) to[bend right] (2,-4);
\draw (1,-1) to[bend left] (1,-3) to[bend left] (1,-1);\draw (1,-1)--+(-.15,.2) (1,-3)--+(-.15,-.2);
\draw (1,1) to[bend right] (1.5,.5);\draw[dashed] (1,1) to[bend left] (1.5,.5);
\draw (1,1) edge[out=80,in=-120] (1.9,2) (2.5,1.4) edge[out=-90,in=0] (1.5,.5) (2.1,2) edge[out=-90,in=91] (2.5,1.4);
\draw (1.9,2) to[bend left] (2.1,2) to[bend left] (1.9,2);
\draw (1.9,1.4) to[bend left] (2.1,1.4) (1.85,1.41) to[bend right] (2.15,1.41) (1.3,1.2) to[bend left] (1.5,1.2) (1.25,1.21) to[bend right] (1.55,1.21) (1.7,1.2) to[bend left] (1.9,1.2) (1.65,1.21) to[bend right] (1.95,1.21) (2.1,1.2) to[bend left] (2.3,1.2) (2.05,1.21) to[bend right] (2.35,1.21) (1.5,1.4) to[bend left] (1.7,1.4) (1.45,1.41) to[bend right] (1.75,1.41) (1.9,1) to[bend left] (2.1,1) (1.85,1.01) to[bend right] (2.15,1.01);
\draw[fill] (1.5,.9) circle (.02) (1.7,.8) circle (.02);
\node at (-1.8,.8) {$\Sigma'$};\node at (2.5,.7) {$\Xi'$};\node at (2.2,-1) {\reflectbox{$\Sigma'$}};
\end{tikzpicture}
\end{center}\end{minipage}
\begin{minipage}{.47\textwidth}
(\textsc{a}) Inclusion $\iota\colon(\Sigma',x)\hookrightarrow(\Sigma,x)$ of $\Sigma'$ of genus 1 with 6 boundary components into $\Sigma$ of genus 4 with 3 punctures. The complement is a torus with boundary $\Sigma^G$, a pair of pants $\Sigma^B$, and once and twice punctured discs $\Sigma^D$, $\Sigma^P$.
\end{minipage}
\hfill
\begin{minipage}{.47\textwidth}
(\textsc{c}) The surface $\Xi$, obtained from $\widehat{\Sigma}$ by retracting punctured discs via $\rho_1$. Collapsing $\Xi'$ to a point via $\rho_2$ gives a double of $\Sigma'$ glued along some boundary components. Resulting $\Sigma'\cup$\reflectbox{$\Sigma'$} then retracts onto $\Sigma'$ via $\rho_3$.
\end{minipage}
\begin{minipage}{\textwidth}
\hfill
\begin{tikzpicture}[scale=.94]
\draw[dashed,fill=black!10!white] (0,2) to[bend left] (0,1) to[bend left] (0,2) (0,-3) to[bend left] (0,-4) to[bend left] (0,-3);
\draw[dash pattern={on 2pt off 2pt},fill=black!10!white] (0,0) to[bend left] (0,-.4) to[bend left] (0,0) (0,-.8) to[bend left] (0,-1.2) to[bend left] (0,-.8) (0,-1.6) to[bend left] (0,-2) to[bend left] (0,-1.6);
\draw[dashed,fill=black!10!white] (-2,-4) to[bend left] (-1,-4) to[bend left] (-2,-4);
\draw[fill=lightgray] (0,0) to[bend left=50] (0,1) to[bend left] (0,2) to[bend right] (-2,-4) to[bend left] (-1,-4)--(0,-4) to[bend left] (0,-3) to[bend left=50] (0,-2) to[bend left] (0,-1.6) ..controls (-.3,-1.6) and (-.3,-1.2).. (0,-1.2) to[bend left] (0,-.8) ..controls (-.3,-.8) and (-.3,-.4).. (0,-.4) to[bend left] (0,0);
\draw[fill=white] (-1,-1) to[bend left] (-1,-3) to[bend left] (-1,-1);\draw (-1,-1)--+(.15,.2) (-1,-3)--+(.15,-.2);
\draw[fill] (-.5,0) circle (.05) node[anchor=south] {$\hat{x}$};
\draw (-1,-4) arc (0:-180:.5);
\draw[fill] (-1.5,-4.35) circle (.05);
\draw (5,-4) arc (-180:0:.5);
\draw[fill] (5.5,-4.35) circle (.05);
\draw (4,0) to[bend right=50] (4,1) to[bend right] (4,2) to[bend left] (6,-4) to[bend right] (5,-4)--(4,-4) to[bend right] (4,-3) to[bend right=50] (4,-2) to[bend right] (4,-1.6) ..controls (4.3,-1.6) and (4.3,-1.2).. (4,-1.2) to[bend right] (4,-.8) ..controls (4.3,-.8) and (4.3,-.4).. (4,-.4) to[bend right] (4,0);
\draw[dashed] (5,-4) to[bend right] (6,-4) (4,1) to[bend left] (4,2);
\draw[dash pattern={on 2pt off 2pt}] (4,-.4) to[bend left] (4,0) (4,-1.2) to[bend left] (4,-.8) (4,-2) to[bend left] (4,-1.6);
\draw (5,-1) to[bend left] (5,-3) to[bend left] (5,-1);\draw (5,-1)--+(-.15,.2) (5,-3)--+(-.15,-.2);
\draw (0,1)--(4,1) (0,2)--(1,2) to[bend right=20] (2,2)--(4,2);
\draw[dashed] (1,2) to[bend left] (2,2);
\draw (.5,1.5) to[bend left=15] (3.5,1.5) (3.65,1.55)--(3.5,1.5) to[bend left=15] (.5,1.5)--(.35,1.55);
\draw[red] (.5,2) to[bend left] (.8,1.58) (3.5,2) to[bend right] (3.2,1.58) (2,1) to[bend right] (2,1.27);
\draw[red,dash pattern={on 2pt off 2pt}] (.5,2) to[bend right] (.8,1.58) (3.5,2) to[bend left] (3.2,1.58) (2,1) to[bend left] (2,1.27);
\draw (0,0) edge[out=0,in=-170] (2,.5) (0,-.4) edge[out=0,in=-170] (2,-.1) (2,.3) ..controls (1.6,.1).. (2,.1);
\draw (3.2,0)--(4,0) (3.2,-.2) edge[out=-60,in=135] (4,-1.2) (4,-.4) ..controls (3.7,-.4).. (4,-.8);
\draw (0,-.8) edge[out=-10,in=180] (4,-1.6) (0,-2)--(4,-2) (0,-1.2) ..controls (.8,-1.45).. (0,-1.6);
\draw[red] (.5,.05) to[bend left] (.6,-.37) (3.1,-1.54) to[bend right] (3,-2) (3.6,0) to[bend left] (3.4,-.5);
\draw[red, dashed] (.5,.05) to[bend right](.6,-.37) (3.1,-1.54) to[bend left] (3,-2) (3.6,0) to[bend right] (3.4,-.5);
\draw[dashed] (4,-4) to[bend left] (4,-3) (1.5,-2)--(1.5,-1.16) (2,-2)--(2,-1.3);
\draw[red] (1.5,-2.7)--(2,-3)--(2.3,-2.7) (2,-3)--(2,-4);\draw[red,dashed] (1.5,-2.7) to[bend right] (2,-4) to[bend right] (2.3,-2.7);
\draw (0,-4)--(4,-4) (0,-3)--(1,-3) edge[out=0,in=-90] (1.5,-2.7) (1.5,-2.7)--(1.5,-2) (2,-2) edge[out=-90,in=135] (2.3,-2.7) (2.3,-2.7) edge[out=-45,in=180] (3,-3) (3,-3)--(4,-3) (1.5,-1.16) edge[out=90,in=180] (2.4,-.5) (2,-1.3) edge[out=90,in=180] (2.4,-1);
\draw[fill] (2,-3) circle (.05) (2,-4) circle (.05);
\draw (1,2) edge[out=45,in=100] (3,2) (2,2) edge[out=45,in=100] (2.5,2);
\draw[dash pattern={on 1pt off 1pt}] (2.5,2)--(2.5,1.7) (2.5,1.3)--(2.5,1) (3,2)--(3,1.61) (3,1.39)--(3,1);
\draw (2.5,1.7)--(2.5,1.3) (3,1.61)--(3,1.39);
\draw (2.5,1) edge[out=-90,in=0] (2,.5) (2,.5) to[bend left] (2,.3) to[bend left=60] (2,.1) to[bend left] (2,-.1) (2,-.1) edge[out=0,in=90] (2.5,-.3) (2.5,-.3) edge[out=-90,in=0] (2.4,-.5) (2.4,-.5) to[bend left] (2.4,-1)--(2.5,-1) to[bend left] (3,-1) (3,-1) to[bend left] (3.2,-.2) to[bend left] (3.2,0) (3.2,0) edge[out=120,in=-90] (3,1);
\draw (2.75,.7) to[bend right] (2.75,.1) (2.73,.8) to[bend left] (2.73,0);
\draw[dash pattern={on 1pt off 1pt}] (2,.5) to[bend right] (2,.3) (2,.1) to[bend right] (2,-.1) (3.2,-.2) to[bend right] (3.2,0);
\draw[dashed] (2.5,-1) to[bend right] (3,-1);
\draw (3,-1) arc (0:-180:.25);
\draw[fill] (2.75,-1.17) circle (.02);
\node at (-1.8,.8) {$\hat{\iota}(\Sigma')$};\node at (6,.8) {$t^2.\hat{\iota}(\Sigma')$};\node at (3.5,.5) {\tiny $t.\hat{\iota}(\Sigma')$};
\node at (-.8,-4.4) {$\widehat{\Sigma}^D$};\node at (4.6,-4.4) {$t^2.\widehat{\Sigma}^D$};\node at (3.1,-1.3) {\tiny $t.\widehat{\Sigma}^D$};
\node at (2,1.5) {\small $\widehat{\Sigma}^G$};
\node at (1.5,-.5) {\small $\widehat{\Sigma}^B$};\node at (3.35,-.95) {\tiny $t.\widehat{\Sigma}^B$};\node at (.6,-2.25) {$t^2.\widehat{\Sigma}^B$};
\node at (3.2,-4.3) {$\widehat{\Sigma}^P$};
\node[red] at (.5,2.3) {$t.\hat{\beta}^G$};\node[red] at (3.5,2.3) {$t^2.\hat{\beta}^G$};\node[red] at (2,.8) {\tiny $\hat{\beta}^G$};\node[red] at (.5,.3) {\small $\hat{\beta}^B$};\node[red] at (3.7,.18) {\tiny $t.\hat{\beta}^B$};\node[red] at (3.4,-2.3) {$t^2.\hat{\beta}^B$};\node[red] at (1,-2.7) {$t^2.\hat{\beta}^P$};\node[red] at (2.6,-2.5) {$\hat{\beta}^P$};\node[red] at (2,-4.3) {$t.\hat{\beta}^P$};
\end{tikzpicture}
\hfill\ \\
(\textsc{b}) The cover $\widehat{\Sigma}$ of $\Sigma$ with deck group $\langle t|t^3\rangle$. Preimages of $\hat{\iota}(\Sigma'),\Sigma^D,\Sigma^B$ consist of 3 disjoint copies of them. Preimages of $\Sigma^G$, $\Sigma^P$ are cyclic covers (the latter is branched). Retracting $\widehat{\Sigma}^D,t.\widehat{\Sigma}^D,t^2.\widehat{\Sigma}^D$ onto their boundaries gives $\Xi$.
\end{minipage}
\caption{Illustration of the construction in Lemma \ref{lmm:virt_retraction}. Indices $i,j,k,l$ are omitted for brevity.}
\label{fig:virt_retraction}
\end{figure}

\begin{proof}[Proof of Lemma \ref{lmm:virt_retraction}]

\smallskip
\textbf{Part 1:} \textit{Covering $F\colon(\widehat{\Sigma},\hat{x})\to(\Sigma,x)$.}
\smallskip

Denote the components of the complement $\Sigma\setminus\Sigma'$ that share at least 2 boundary components with $\Sigma'$ by $\Sigma^B_i$. Divide the remaining ones into once-punctured discs $\Sigma^D_j$, multiple-punctured components $\Sigma^P_k$, and all other $\Sigma^G_l$. By assumption, every $\Sigma^G_l$ has positive genus. An example of this setting is depicted on Figure \ref{fig:virt_retraction}(\textsc{a}).

For each $\Sigma^B_i$, pick one of its boundary components $\beta^B_i$ (perturbed away from $\Sigma'$). On each $\Sigma^P_k$ pick a simple arc $\beta^P_k$ between two distinct punctures and disjoint from $\Sigma'$. For each $l$, pick some essential nonseparating simple closed curve $\beta^G_l$ on $\Sigma^G_l$ that is disjoint from $\Sigma'$; it exists because their genera are positive. Endow each $\beta_m\in\{\beta^B_i\}\cup\{\beta^P_k\}\cup\{\beta^G_l\}$ with an arbitrary orientation, and let $\Theta_m\colon\pi_1\Sigma\to\tfrac{\mathbb{Z}}{d\mathbb{Z}}$ be the intersection number with $\beta_m$ modulo $d$. All such $\beta_m$ are nonseparating in $\Sigma$, so each $\Theta_m$ is surjective. Define $\Theta=\sum_m\Theta_m$, and take $(\widehat{\Sigma},\hat{x})$ to be the covering space associated to $\ker\Theta$.

We know that $\Theta$ is surjective, because we said each $\Theta_m$ is (and there is at least one $m$ by non-degeneracy assumptions). Therefore $\widehat{\Sigma}$ is a Galois covering with deck group $\tfrac{\mathbb{Z}}{d\mathbb{Z}}$. Each $\beta_m$ is disjoint from $\iota(\Sigma')$, so $\Theta$ vanishes on $\pi_1(\Sigma',x)$. This means that the original inclusion $\iota\colon(\Sigma',x)\hookrightarrow(\Sigma,x)$ lifts to inclusion $\hat{\iota}\colon(\Sigma',x)\hookrightarrow(\widehat{\Sigma},\hat{x})$ into the covering space. Further, the preimage $F^{-1}(\Sigma')$ of $\Sigma'$ in $\widehat{\Sigma}$ consists of $d$ disjoint copies of $\Sigma'$. The cover $\widehat{\Sigma}$ with $d=3$ for an earlier example is shown in Figure \ref{fig:virt_retraction}(\textsc{b}).

\smallskip
\textbf{Part 2:} \textit{Retraction $\rho\colon\widehat{\Sigma}\to\Sigma'$.}
\smallskip

First, the preimage of each $\Sigma^D_j$ is a disjoint union of $d$ copies of it. Let $\rho_1\colon\widehat{\Sigma}\to\Xi$ retract each of them into its boundary. By abuse of notation, we treat $\hat{\iota}(\Sigma')$ also as a subsurface of $\Xi$.

Second, $\Xi\setminus\hat{\iota}(\Sigma')$ is connected. Indeed, its every point can be connected to one of the $d-1$ components of $F^{-1}(\Sigma')\setminus\hat{\iota}(\Sigma')$ by a path disjoint from $\hat{\iota}(\Sigma')$, each of these components is connected, and any two of them can again be joined by a path disjoint from $\hat{\iota}(\Sigma')$.

Third, $\Xi\setminus\hat{\iota}(\Sigma')$ contains at least one embedded copy of $\Sigma'$. Therefore its genus and number of punctures is at least that of $\Sigma'$. Similarly, $\Xi\setminus\hat{\iota}(\Sigma')$ contains at least as many components of $\partial\Xi$ as $\hat{\iota}(\Sigma')$ does (these came from the retraction $\rho_1$ earlier). By the classification of surfaces, $\Xi\setminus\hat{\iota}(\Sigma')$ is homeomorphic to a connect sum of a copy of $\Sigma'$, which we denote by \reflectbox{$\Sigma'$} (read ``mirror of $\Sigma'$''), and some other surface $\Xi'$. Furthermore, this homeomorphism is compatible with the identification of the common boundaries with $\hat{\iota}(\Sigma')$. Figure \ref{fig:virt_retraction}(\textsc{c}) shows surface $\Xi$ from this viewpoint.

Finally, let $\rho_2\colon\Xi\to\hat{\iota}(\Sigma')\cup$\reflectbox{$\Sigma'$} collapse $\Xi'$ to a point. Then let $\rho_3\colon\hat{\iota}(\Sigma')\cup$\reflectbox{$\Sigma'$}$\to\Sigma'$ be identity on $\hat{\iota}(\Sigma')$ and send \reflectbox{$\Sigma'$} to $\Sigma'$.

We define $\rho=\rho_3\circ\rho_2\circ\rho_1$. This is a composition of continuous functions, so is continuous. Furthermore, each of $\rho_1,\rho_2,\rho_3$ restricted to the suitable included copy of $\Sigma'$ is a homeomorphism onto its image, so $\rho\circ\hat{\iota}$ is the identity on $\Sigma'$.
\end{proof}

\subsection{Proof of Theorem \ref{thm:csp_subsurface}}\label{ssec:pf_csp_subsurface}

Before the full Theorem \ref{thm:csp_subsurface} we prove its weakening Proposition \ref{prop:subsurface_containment}.
\begin{proposition}\label{prop:subsurface_containment}
Suppose $\Sigma'\hookrightarrow\Sigma$ is an inclusion of subsurfaces such that no component of $\Sigma\setminus\Sigma'$ is a disc with zero or one puncture (and this complement is non-empty). Assume $\Lambda\leq\Mcg(\Sigma')$ is a subgroup such that $\Lambda\to\Out(\pi_1\Sigma')$ has the CSP. Then for any finite-index subgroup $\Gamma\leq\Lambda$, such that $\Lambda\setminus\Gamma$ contains no multitwists about $\partial\Sigma'$, there exists a finite characteristic quotient $q\colon\pi_1\Sigma\onto Q$ such that all mapping classes from $(\Lambda\setminus\Gamma)\cdot\Mcg\left(\Sigma\setminus\Sigma'\right)$ are non-trivial in $\Out(Q)$.
\end{proposition}
\begin{proof}

\smallskip
\textbf{Step 1:} \textit{Construction of the quotient $(q,Q)$.}
\smallskip

Let $x$ be an auxiliary basepoint in the interior of $\Sigma'$. Note that the kernel of $\Lambda\to\Out(\pi_1\Sigma')$ consists precisely of boundary twists, so is contained in $\Gamma$ by assumption. By another assumption, $\Lambda\to\Out(\Sigma')$ has CSP, so there is a finite characteristic quotient $r\colon\pi_1(\Sigma',x)\onto R$ such that $\ker\left(r_\star|_\Lambda\right)$ is contained in $\Gamma$.

Pick integer $d\geqslant 2$ coprime to $|\Out(R)|$. Lemma \ref{lmm:virt_retraction} tells us that $\pi_1(\Sigma',x)$ sits as a retract inside a normal subgroup $\pi_1(\widehat{\Sigma},\hat{x})\triangleleft\pi_1(\Sigma,x)$ whose quotient is $\tfrac{\mathbb{Z}}{d\mathbb{Z}}$. Denote the retraction map by $\rho\colon\pi_1(\widehat{\Sigma},\hat{x})\to\pi_1(\Sigma',x)$. The composition $r \circ \rho$ makes $R$ into a quotient of the finite-index subgroup $\pi_1(\widehat{\Sigma},\hat{x})$ of $\pi_1(\Sigma,x)$. Induce and refine it to a finite characteristic quotient $q \colon \pi_1(\Sigma,x) \rightarrow Q$.

Denote the image of $\pi_1(\widehat{\Sigma},\hat{x})$ under $q$ by $\widehat{Q}$. So far we have produced the following commutative diagram.
\begin{equation}\label{eq:virt_subsurface_retraction}\begin{tikzcd}[row sep=1ex]
\pi_1(\Sigma',x)\arrow[dd,two heads,"r"]\arrow[r,hook] & \pi_1\left(\widehat{\Sigma},\hat{x}\right)\arrow[r,hook,"F"]\arrow[dd,two heads,"q"]\arrow[ddl,two heads,"r\circ\rho"] & \pi_1\left(\Sigma,x\right)\arrow[rd,two heads,bend left,"\Theta"]\arrow[dd,two heads,"q"] &\\
& & & \frac{\mathbb{Z}}{d\mathbb{Z}}\\
R & \widehat{Q}\arrow[r,hook,"f"]\arrow[l,two heads] & Q\arrow[ru,two heads,bend right,"\theta"] &
\end{tikzcd}\end{equation}
Moreover, $\text{im}\, F=\ker\Theta$ and $\text{im}\, f=\ker\theta$.

\smallskip
\textbf{Step 2:} \textit{Action of $\Lambda\cdot\Mcg(\Sigma\setminus\Sigma')$ on $(q,Q)$.}
\smallskip

Consider $\phi\in\Lambda_i$, $\psi\in\Mcg(\Sigma\setminus\Sigma')$ and suppose that the mapping class $\phi\psi$ becomes trivial in $\Out(Q)$. Proposition \ref{prop:subsurface_containment} will be proven if we show that $\phi\in\Gamma$. Pick representatives $\Phi\in\tn{Homeo}(\Sigma')$, $\Psi\in\tn{Homeo}(\Sigma\setminus\Sigma')$ which preserve $\partial\Sigma'$ pointwise and such that $\Phi$ preserves the basepoint $x$. Then $\Phi\Psi$ induces an (honest) automorphism of the fundamental group $\pi_1(\Sigma,x)$; this automorphism becomes inner $\ad_g\in\text{Aut}(Q)$ for some $g\in Q$ when we pass to the quotient $Q$.

The first problem we face is the interference from $\Psi$. So let us examine what happens to the diagram (\ref{eq:virt_subsurface_retraction}) once we apply $\Phi\Psi$. The inner automorphism $\text{ad}_g$ must preserve the homomorphism $\theta$, because the target is abelian. This means that $\Phi\Psi$ must preserve $\ker\theta=\widehat{Q}$ inside $Q$, the homomorphism $\Theta$, and the subgroup $\ker\Theta=\pi_1(\widehat{\Sigma},\hat{x})$. This gives us the following commutative diagram.
\begin{equation*}\begin{tikzcd}
\pi_1(\Sigma',x)\arrow[r]\arrow[d,"\Phi"] & \pi_1\left(\widehat{\Sigma},\hat{x}\right)\arrow[r]\arrow[d,"\Phi\Psi"] & \widehat{Q}\arrow[d,"\textnormal{ad}_g"] \\
\pi_1(\Sigma',x)\arrow[r] & \pi_1\left(\widehat{\Sigma},\hat{x}\right)\arrow[r] & \widehat{Q}
\end{tikzcd}\end{equation*}
Crucially, the leftmost arrow contains only $\Phi$. This is because $\Psi$ is supported on $\Sigma\setminus\Sigma'$, so it acts trivially on $\pi_1(\Sigma', x)$.

The second problem is that $\text{ad}_g$ need not be inner in $\text{Aut}(\widehat{Q})$ if $g\notin\widehat{Q}$. Thus we take the $d$-th power, $\theta(g^d)=0$ so $g^d\in\widehat{Q}$. Let $h$ be the image of $g^d$ in $R$. We obtain the following diagram.
\begin{equation*}\begin{tikzcd}
\pi_1(\Sigma',x)\arrow[r]\arrow[d,"\Phi^d"]\arrow[rr,two heads,bend left,"r"] & \widehat{Q}\arrow[r]\arrow[d,"\text{ad}_{g^d}"] & R\arrow[d,"\text{ad}_h"] \\
\pi_1(\Sigma',x)\arrow[r]\arrow[rr,two heads,bend right,"r"] & \widehat{Q}\arrow[r] & R
\end{tikzcd}\end{equation*}

Therefore, when we pass to the quotient $r\colon\pi_1(\Sigma',x)\onto R$, the automorphism $\Phi^d$ becomes inner in $\Aut(R)$, or equivalently the image of the mapping class $\phi^d$ in $\Out(R)$ is trivial. Since we chose $d$ coprime to $|\Out(R)|$, this means that $\phi$ must itself be trivial in $\Out(R)$. But by the choice of $R$ we have $\ker r\cap\Lambda\subseteq\Gamma$, so $\phi\in\Gamma$ as desired.
\end{proof}

We are ready to complete the proof of Theorem \ref{thm:csp_subsurface}. This involves combining Proposition \ref{prop:subsurface_containment} with Lemmata \ref{lmm:csp_multitwist} and \ref{lmm:csp_ses}. We recall the statement for the reader's convenience.
\begin{theorem*}[restatement of \ref{thm:csp_subsurface}]
Let $\Sigma$ be a possibly punctured surface with no boundary, and let $\Sigma=\bigcup_i\Sigma_i$ be its decomposition into disjoint subsurfaces glued along boundaries. Suppose that $\Lambda_i\leq\Mcg(\Sigma_i)$ are subgroups supported on the subsurfaces, and let $\Lambda=\prod_i\Lambda_i$ be the subgroup of $\Mcg(\Sigma)$ obtained from them by inclusion. If all $\Lambda_i\to\Out(\pi_1\Sigma_i)$ have the CSP, then $\Lambda\to\Out(\pi_1\Sigma)$ has the CSP.
\end{theorem*}
\begin{proof}[Proof of Theorem \ref{thm:csp_subsurface}]
Let $\Delta'$ be the subgroup of $\Mcg(\Sigma)$ generated by the Dehn twists about boundaries of subsurfaces $\Sigma_i$. Denote $\Delta=\Delta'\cap\Lambda$. By Lemma \ref{lmm:csp_multitwist} we know that $\Delta\to\Out(\pi_1\Sigma)$ has CSP.

Consider any finite-index $\Gamma\leq\Lambda$ containing $\Delta$. By Lemma \ref{lmm:csp_ses}, it is enough to show that $\Gamma$ is a congruence subgroup of $\Lambda$.

Define $\Gamma_i = \Gamma\cap\Lambda_i$. This is a finite-index subgroup of $\Lambda_i$, and every boundary multitwist contained in $\Lambda_i$ is also an element of $\Gamma_i$. By Proposition \ref{prop:subsurface_containment}, there exists a finite characteristic quotient $q_i\colon\pi_1\Sigma\onto Q_i$ such that no mapping class in $(\Lambda_i\setminus\Gamma_i)\cdot\Mcg(\Sigma\setminus\Sigma_i)$ is trivial in $\Out(Q_i)$.

Consider the product $\prod_iq_i\colon\pi_1\Sigma\to\prod_iQ_i$ of the quotients $(q_i,Q_i)$. Let $Q$ be its image, and $q\colon\pi_1\Sigma\onto Q$ the corresponding factorisation.

Suppose $\varphi\in\Lambda$ becomes trivial in $\Out(Q)$. Since $q_i$ factors through $q$, we know that $q_{i,\star}(\varphi)$ is trivial in $\Out(Q_i)$ for each $i$. Write $\varphi=\prod_i\varphi_i$ for $\varphi_i\in\Lambda_i$. From the main property of $Q_i$ we know that $\varphi_i\in\Gamma_i$. Therefore $\varphi\in\prod_i\Gamma_i\subseteq\Gamma$. This means that $\Gamma$ contains the principal congruence subgroups induced by $(q,Q)$, so it is a congruence subgroup.
\end{proof}

\section{Example application: handle pushing subgroups}\label{sec:handlepushing}

The purpose of this section is to introduce handle-pushing subgroups, and explain how they fit in the framework of Theorem \ref{thm:csp_subsurface}. We start from the algebraic side.
\begin{definition}
Denote a closed surface of genus $g$ by $\Sigma_g$, its unordered configuration space of $p$ points by $\tn{Conf}_p$, and its unit tangent bundle by $\tn{UT}\Sigma_g$. Let $C\defeq\tn{Conf}_p(\Sigma_g)\times\left(\tn{UT}\Sigma_g\right)^b$ (this can be thought of as a mix between ordered and unordered configuration spaces, where each of the ``ordered/distinguished'' points carries an attached tangent direction). Define the ``diagonal'' subset as
\begin{equation*}
D\defeq\left\{\big(\{x_1,\dots,x_p\},(x_{p+1},v_{p+1}),\dots,(x_{p+b},v_{p+b})\big)\in C\ \big|\ x_i=x_j\tn{ for some }i\neq j\right\}.
\end{equation*}
The \emph{framed genus-$g$ braid group on $p$ ordinary strands and $b$ distinguished strands}, denoted $B_{g,p,b}$, is the fundamental group of $C\setminus D$.
\end{definition}
Elements of $B_{g,p,b}$ correspond to collisionless motions of $p+b$ points on the surface $\Sigma_g$. A subset of $b$ of points is distinguished, and each distinguished point has a tangent direction attached to itself, which moves continuously with the point (this is sometimes referred to as \emph{framing}). Note that having two distinguished points with the same basepoint but different vectors still counts as a collision and is disallowed. See Figure \ref{fig:example_braid} for an example element of $B_{2,3,1}$. The treatment of framed braid groups in this article only covers the few aspects necessary for our results, for more a detailed discussion see \cite{surf_framed_braids_BellingeriGervais12}.

Forgetting the intermediate movement and only remembering the permutation of the ordinary points gives a surjection $B_{g,p,b}\onto S_p$ onto the symmetric group of $p$ elements. Its kernel is often referred to as \emph{pure} braid group. Rotating the tangent directions while keeping all points fixed gives a copy of $\mathbb{Z}^b$ inside $B_{g,p,g}$, which is central. Forgetting the tangent vectors (framing) turns distinguished strands into ordinary strands. Algebraically, this corresponds to a homomorphism $B_{g,p,b}\to B_{g,p+b,0}$, whose kernel is the aforementioned central copy of $\mathbb{Z}^b$, and whose image coincides with the preimage of $S_p$ under the homomorphism $B_{g,p+b,0}\to S_{p+b}$.

\begin{figure}[h]
\begin{tikzpicture}
\draw plot [smooth cycle] coordinates {(0,-.5) (1,-1) (2,-1) (3,0) (3,1) (2,2) (1,2) (0,1.5) (-1,2) (-3,2) (-4,1) (-4,0) (-3,-1) (-1,-1)};
\draw (1,.5) to[bend left] (2,.5) (.9,.6) to[bend right=60] (2.1,.6);
\draw (-2,.5) to[bend left] (-1,.5) (-2.1,.6) to[bend right=60] (-.9,.6);
\node at (-3.2,1.2) {$a$};\node at (-3.5,.5) {$b$};\node at (-3.2,-.2) {$c$};\node at (-.3,.5) {$d$};
\draw[blue] (-3,1) .. controls (-4.3,-1.5) and (-4.3,2.5) .. (-3,0);
\draw[teal] (-3,0) to[bend right=20] (-1.5,-1.05);\draw[dashed,teal] (-1.5,-1.05)--(-1,.5);\draw[teal] (-1,.5) to[bend right=60] (-3,1);
\node[fill=blue,regular polygon,regular polygon sides=3,inner sep=1.5] at (-3.96,.5) {};
\node[fill=teal,regular polygon,regular polygon sides=3,inner sep=1.5] at (-2.3,-.72) {};
\node[fill=teal,regular polygon,regular polygon sides=3,inner sep=1.5,shape border rotate=80] at (-1.9,1.26) {};
\draw[red] (0,.5) .. controls (3.5,3.3) and (3.5,-2.3) .. (0,.5);
\node[fill=red,regular polygon,regular polygon sides=3,inner sep=1.5,shape border rotate=30] at (1.8,1.3) {};
\node[fill=red,regular polygon,regular polygon sides=3,inner sep=1.5,shape border rotate=90] at (1.8,-.31) {};
\draw[->] (0,.5)--(0,.8);\draw[red,->] (.5,.86)--(.35,1.05);\draw[red,->] (1,1.13)--(.7,1.1);\draw[red,->] (1.4,1.27)--(1.2,1);\draw[red,->] (2,1.28)--(2,1);\draw[red,->] (2.46,1)--(2.5,.7);\draw[red,->] (2.63,.5)--(2.7,.2);\draw[red,->] (2.46,0)--(2.7,-.1);\draw[red,->] (2,-.28)--(2.35,-.25);\draw[red,->] (1.4,-.27)--(1.65,-.2);\draw[red,->] (1,-.13)--(1.2,.1);\draw[red,->] (.5,.14)--(.6,.4);
\draw [fill] (0,.5) circle (.05) (-3,1) circle (.05) (-3.7,.5) circle (.05) (-3,0) circle (.05);
\node[blue] at (-3.6,-.1) {$\alpha$};\node[teal] at (-1.5,-.3) {$\gamma$};\node[red] at (1.7,-.7) {$\delta$};
\end{tikzpicture}
\caption{An example element of a framed genus 2 braid group on 3 ordinary strands ($a,b,c$) and 1 distinguished strand ($d$). Ordinary points $a$ and $c$ travel along {\color{blue}$\alpha$} and {\color{teal}$\gamma$} respectively, and distinguished point $d$ together with its tangent vector travel along {\color{red}$\delta$}, in the direction indicated by solid triangles. The image of this braid in $S_3$ is the transposition swapping $a,c$. Note that the motion of $d$ is homotopic to rotating the vector counterclockwise while keeping the point in place, followed by moving $d$ across $\delta$ while keeping the direction of the vector constantly pointing up (in the plane of the picture).}
\label{fig:example_braid}
\end{figure}
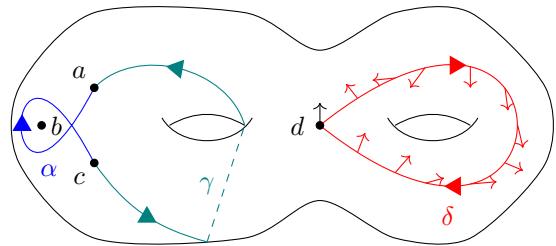

Each braid induces a mapping class, on a surface where ordinary points are replaced with punctures and distinguished points are replaced with boundaries.
\begin{definition}\label{def:pt_push}
Let $\Sigma_{g,p,b}$ be a surface of genus $g$ with $p$ punctures and $b$ boundary components. The \emph{point-pushing homomorphism} is the map $B_{g,p,b}\to\Mcg(\Sigma_{g,p,b})$ obtained by ``pushing'' the punctures along ordinary strands, ``pushing'' boundary components along the distinguished strands, in a way that the tangent direction keeps track of the rotation of the boundary component, and ``dragging'' the rest of the surface along.
\end{definition}
For the sake of brevity, we will not make Definition \ref{def:pt_push} any more formal, and instead just draw a pictorial example in Figure \ref{fig:ptpush_example}(\textsc{a,b}). For a rigorous treatment consult the original work \cite{mcgs_relationship_braid_gps_Birman69} or newer surveys \cite[Sections 4.2 and 3.6]{primer_on_mapping_class_groups_FarbMargalit12} and \cite[Section 1.3]{braids_survey_BirmanBrendle05}. The reader might find it helpful to convince themselves that point-pushing along a simple loop is the same as a product of two Dehn twists, as illustrated on Figure \ref{fig:ptpush_example}(\textsc{c}).

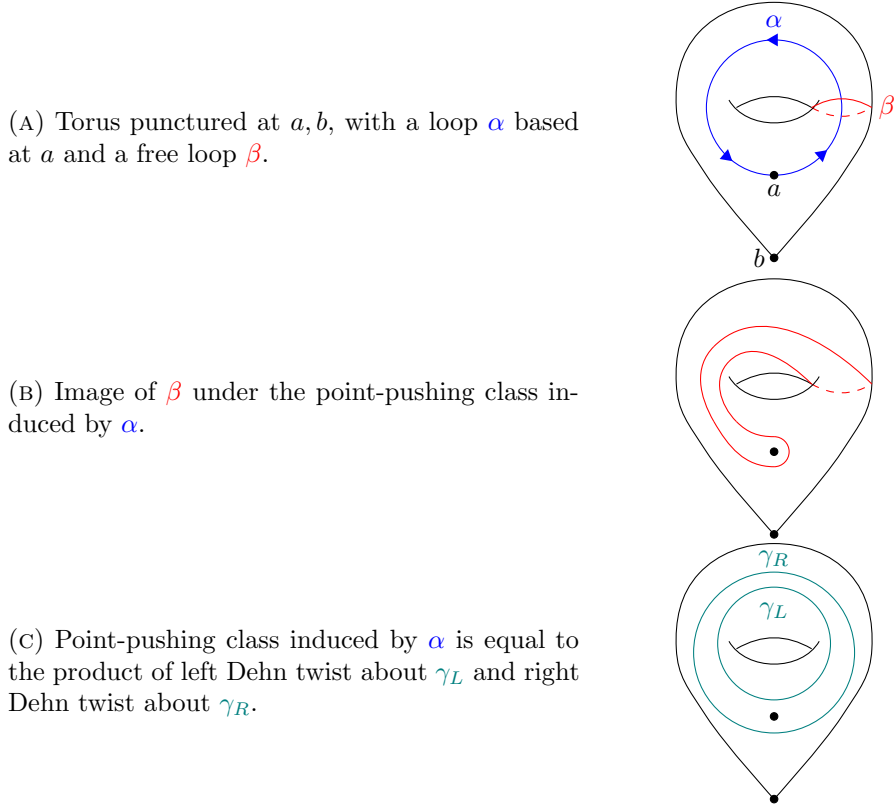
\begin{figure}[h]
\begin{minipage}{.6\textwidth}
(\textsc{a}) Torus punctured at $a,b$, with a loop {\color{blue}$\alpha$} based at $a$ and a free loop {\color{red}$\beta$}.
\end{minipage}
\hfill
\begin{minipage}{.3\textwidth}
\begin{tikzpicture}
\draw plot [smooth,tension=.7] coordinates {(0,-1) (.9,.1) (1.3,1) (1,2) (0,2.4) (-1,2) (-1.3,1) (-.9,.1) (0,-1)};
\draw[red,dashed] (.5,1) to[bend right] (1.3,1);
\draw (-.5,1) to[bend left] (.5,1) (-.6,1.1) to[bend right=60] (.6,1.1);
\draw[blue] (0,1) circle (.9);
\node[fill=blue,regular polygon,regular polygon sides=3,inner sep=1,shape border rotate=70] at (.63,.37) {};
\node[fill=blue,regular polygon,regular polygon sides=3,inner sep=1,shape border rotate=90] at (0,1.9) {};
\node[fill=blue,regular polygon,regular polygon sides=3,inner sep=1,shape border rotate=-10] at (-.63,.35) {};
\draw[red] (.5,1) to[bend left] (1.3,1);
\draw[fill] (0,.1) circle (.05) node[anchor=north] {$a$};\draw[fill] (0,-1) circle (.05) node[anchor=east] {$b$};
\node[blue] at (0,2.15) {$\alpha$};\node[red] at (1.5,1) {$\beta$};
\end{tikzpicture}
\end{minipage}
\begin{minipage}{.6\textwidth}
(\textsc{b}) Image of {\color{red}$\beta$} under the point-pushing class induced by {\color{blue}$\alpha$}.
\end{minipage}
\hfill
\begin{minipage}{.3\textwidth}
\begin{tikzpicture}
\draw plot [smooth,tension=.7] coordinates {(0,-1) (.9,.1) (1.3,1) (1,2) (0,2.4) (-1,2) (-1.3,1) (-.9,.1) (0,-1)};
\draw[red,dashed] (.5,1) to[bend right] (1.3,1);
\draw (-.5,1) to[bend left] (.5,1) (-.6,1.1) to[bend right=60] (.6,1.1);
\draw[fill] (0,.1) circle (.05);\draw[fill] (0,-1) circle (.05);
\draw[red] (.5,1) edge[out=135,in=45] (-.6,1.3);\draw[red] (-.6,1.3) edge[out=225,in=135] (-.5,.5); \draw[red] (-.5,.5) edge[out=-45,in=180] (0,.3);\draw[red] (0,.3) arc(90:-90:.2);\draw[red] (0,-.1) edge[out=180,in=-60] (-.8,.5);\draw[red] (-.8,.5) edge[out=120,in=225] (-.8,1.5);\draw[red] (-.8,1.5) edge[out=45,in=135] (1.3,1);
\end{tikzpicture}
\end{minipage}
\begin{minipage}{.6\textwidth}
(\textsc{c}) Point-pushing class induced by {\color{blue}$\alpha$} is equal to the product of left Dehn twist about {\color{teal}$\gamma_L$} and right Dehn twist about {\color{teal}$\gamma_R$}.
\end{minipage}
\hfill
\begin{minipage}{.3\textwidth}
\begin{tikzpicture}
\draw plot [smooth,tension=.7] coordinates {(0,-1) (.9,.1) (1.3,1) (1,2) (0,2.4) (-1,2) (-1.3,1) (-.9,.1) (0,-1)};
\draw (-.5,1) to[bend left] (.5,1) (-.6,1.1) to[bend right=60] (.6,1.1);
\draw[fill] (0,.1) circle (.05);\draw[fill] (0,-1) circle (.05);
\draw[teal] (0,1.07) circle (.75);\draw[teal] (0,.95) circle (1.07);
\node[teal] at (0,1.5) {$\gamma_L$};\node[teal] at (0,2.2) {$\gamma_R$};
\end{tikzpicture}
\end{minipage}
\caption{Example of a point-pushing map.}
\label{fig:ptpush_example}
\end{figure}

All of the earlier-mentioned algebraic relations have a topological interpretation. A braid and its mapping class induce the same permutation of the punctures, and the pure braid group is the preimage of the pure Mapping Class Group under the point-pushing map. Rotations of tangent vectors become the Dehn twists about the corresponding boundary components, which form a central $\mathbb{Z}^b\leq\Mcg(\Sigma_{g,p,b})$. This subgroup is precisely the kernel of $\Mcg(\Sigma_{g,p,b})\to\Mcg(\Sigma_{g,p+b,0})$ obtained by capping each boundary component with a once-punctured disc.

Additionally, the point-pushing subgroup is precisely the kernel of the homomorphism $\Mcg(\Sigma_{g,p,b})\to\Mcg(\Sigma_{g,0,0})$ obtained by filling the punctures and capping the boundary components with discs. This statement is known as the Birman exact sequence \cite[Theorem 1]{mcgs_relationship_braid_gps_Birman69}\cite[Theorem 4.6]{primer_on_mapping_class_groups_FarbMargalit12}.

\begin{figure}[h]
\begin{minipage}{.47\textwidth}
\begin{tikzpicture}
\draw (-2,1) to[bend left] (-2,0) (-1.9,1.1) to[bend right=60] (-1.9,-.1);
\draw (0,1) to[bend left] (0,0) (.1,1.1) to[bend right=60] (.1,-.1);
\draw (2,1) to[bend left] (2,0) (2.1,1.1) to[bend right=60] (2.1,-.1);
\draw (0,.5) ellipse (3 and 1.5);
\draw[dashed] (2,1.62) to[bend left] (2,1);\draw (2,1.62) to[bend right] (2,1);
\draw[blue] (1.93,1.2) edge[out=220,in=0] (0,-.5);\draw[blue] (0,-.5) edge[out=180,in=-60] (-1.85,.5);\draw[blue,dashed] (-1.85,.5) edge[out=60,in=180] (2.07,1.39);
\node[fill=blue,regular polygon,regular polygon sides=3,inner sep=1,shape border rotate=-30] at (0,-.5) {};
\node[draw=blue,fill=white,regular polygon,regular polygon sides=3,inner sep=1,shape border rotate=30] at (0,1.47) {};
\node[blue] at (.8,.1) {$\alpha$};\node at (2.05,1.9) {$\beta$};\node at (-2.05,1.9) {$\Sigma_3$};
\end{tikzpicture}
(\textsc{a}) Curve $\beta$ cuts genus-3 surface $\Sigma_3$ into a surface of genus 2 with two boundary components. We then have ${\color{blue}\alpha}\in B_{2,0,2}$ which ``moves'' the pair (point,vector) representing one of these components.
\end{minipage}
\hfill
\begin{minipage}{.47\textwidth}
\begin{tikzpicture}
\draw (-2,1) to[bend left] (-2,0) (-1.9,1.1) to[bend right=60] (-1.9,-.1);
\draw (0,1) to[bend left] (0,0) (.1,1.1) to[bend right=60] (.1,-.1);
\draw (2,1) to[bend left] (2,0) (2.1,1.1) to[bend right=60] (2.1,-.1);
\draw (0,.5) ellipse (3 and 1.5);
\draw[dashed] (2,1.62) to[bend left] (2,1);\draw (2,1.62) to[bend right] (2,1);
\draw[teal] (1.7,1.73) edge[out=270,in=0] (0,-.3);\draw[teal] (0,-.3) edge[out=180,in=-60] (-1.87,.7);\draw[teal,dashed] (-1.87,.7) edge[out=60,in=180] (1.7,1.73);
\draw[teal] (1.82,.7) edge[out=250,in=0] (0,-.7);\draw[teal] (0,-.7) edge[out=180,in=-60] (-1.87,.3);\draw[teal,dashed] (-1.87,.3) edge[out=60,in=180] (0,1.4);\draw[teal,dashed] (0,1.4) edge[out=0,in=135] (1.82,.7);
\node at (2.05,1.9) {$\beta$};\node[teal] at (.8,.3) {$\gamma_R$};\node[teal] at (1.5,-.4) {$\gamma_L$};
\end{tikzpicture}
(\textsc{b}) Handle-pushing along {\color{blue}$\alpha$} is the same as right Dehn twist about {\color{teal}$\gamma_R$} and left twist about {\color{teal}$\gamma_L$}, up to some power of Dehn twist about $\beta$ which depends on the number of revolutions the tangent vector does along the way.
\end{minipage}

\vspace{10pt}

\begin{minipage}{.99\textwidth}
\begin{tikzpicture}
\draw (-2,1) to[bend left] (-2,0) (-1.9,1.1) to[bend right=60] (-1.9,-.1);
\draw (0,1) to[bend left] (0,0) (.1,1.1) to[bend right=60] (.1,-.1);
\draw (2,1) to[bend left] (2,0) (2.1,1.1) to[bend right=60] (2.1,-.1);
\draw (0,.5) ellipse (3 and 1.5);
\draw[red] (2,.5) ellipse (.5 and .9);
\end{tikzpicture}
\hfill
\begin{tikzpicture}
\draw (-2,1) to[bend left] (-2,0) (-1.9,1.1) to[bend right=60] (-1.9,-.1);
\draw (0,1) to[bend left] (0,0) (.1,1.1) to[bend right=60] (.1,-.1);
\draw (2,1) to[bend left] (2,0) (2.1,1.1) to[bend right=60] (2.1,-.1);
\draw (0,.5) ellipse (3 and 1.5);
\draw[red] (2.02,1.6) edge[out=-45,in=90] (2.5,.5) (2.5,.5) edge[out=-90,in=20] (2,-.3) (2,-.3) edge[out=200,in=-60] (-1.84,.5);\draw[red,dashed] (-1.84,.5) edge[out=60,in=180] (2.02,1.6);
\end{tikzpicture}

\vspace{5pt}

\begin{tikzpicture}
\draw (-2,1) to[bend left] (-2,0) (-1.9,1.1) to[bend right=60] (-1.9,-.1);
\draw (0,1) to[bend left] (0,0) (.1,1.1) to[bend right=60] (.1,-.1);
\draw (2,1) to[bend left] (2,0) (2.1,1.1) to[bend right=60] (2.1,-.1);
\draw (0,.5) ellipse (3 and 1.5);
\draw[dashed,red] (1,1.91) to[bend left] (1,-.91);\draw[red] (1,1.91) to[bend right] (1,-.91);
\end{tikzpicture}
\hfill
\begin{tikzpicture}
\draw (-2,1) to[bend left] (-2,0) (-1.9,1.1) to[bend right=60] (-1.9,-.1);
\draw (0,1) to[bend left] (0,0) (.1,1.1) to[bend right=60] (.1,-.1);
\draw (2,1) to[bend left] (2,0) (2.1,1.1) to[bend right=60] (2.1,-.1);
\draw (0,.5) ellipse (3 and 1.5);
\draw[dashed,red] (1,-.91)--(1.8,.37) (-1.87,.7) edge[out=60,in=180] (1.4,1.8) (-1.87,.3) edge[out=60,in=180] (0,1.3) (0,1.3) edge[out=0,in=150] (1.89,.9);
\draw[red] (1,-.91) edge[out=170,in=-60] (-1.87,.3) (1.8,.37) edge[out=-160,in=0] (0,-.3) (0,-.3) edge[out=180,in=-45] (-1.87,.7) (1.89,.9)--(1.4,1.8);
\end{tikzpicture}
(\textsc{c}) Some example curves on $\Sigma_3$ (left) and their images under handle-slide induced by {\color{blue}$\alpha$} (right). Note that the image of the top curve depends on the number of revolutions of the tangent vector (which would alter it by a power of Dehn twist about $\beta$).
\end{minipage}
\caption{Example handle slide, a mapping class not found in point-pushing subgroups.}
\label{fig:handle_slide}
\end{figure}

This brings us to the main characters of this section -- the handle-pushing subgroups.
\begin{definition}
Let $B$ be a collection of pairwise non-homotopic essential simple closed curves, and $\Sigma=\bigcup_i\Sigma_i$ be the decomposition into subsurfaces obtained by cutting $\Sigma$ along $B$. The \emph{handle-pushing subgroup} associated to $B$ is the subgroup of $\Mcg(\Sigma)$ generated by inclusions of point-pushing classes from $\Sigma_i$.
\end{definition}
An example of a new mapping class which does not appear in point-pushing subgroups is a \emph{handle slides} illustrated on Figure \ref{fig:handle_slide}. These come from inclusions of pushing boundary components of subsurfaces.

Algebraically, a handle-pushing subgroup is isomorphic to the product of the corresponding point-pushing subgroups, with the identification of the (central) Dehn twists on the two sides of each boundary component, analogously to the situation in Section \ref{sec:pf_csp_subsurface} and the setup of Theorem \ref{thm:csp_subsurface}.

Point-pushing subgroups have CSP by a result of McReynolds \cite{congruence_subgroup_problem_for_pure_braid_groups_McReynolds10}. Theorem \ref{thm:csp_subsurface} immediately applies, giving the CSP of handle-pushing subgroups as Corollary \ref{corr:csp_handlepush}.

\bibliographystyle{alpha}
\bibliography{references}

\end{document}